\documentclass[journal]{IEEEtran}
\usepackage{amsmath}
\usepackage{amssymb}
\usepackage{amsthm}
\usepackage{graphicx}
\usepackage{epstopdf}
\usepackage{cite}
\usepackage{algorithm}
\usepackage{algpseudocode}
\usepackage{tikz}
\usepackage{pgfplots}
\usepackage{subfigure}
\usepgfplotslibrary{groupplots}

\IEEEoverridecommandlockouts

\theoremstyle{plain}
\newtheorem{theorem}{Theorem}

\newtheorem{lemma}[theorem]{Lemma}
\newtheorem{proposition}[theorem]{Proposition}

\newtheorem{assumption}[theorem]{Assumption}
\theoremstyle{definition}

\theoremstyle{remark}

\DeclareMathOperator{\tr}{tr}
\DeclareMathOperator*{\argmin}{arg\,min}
\DeclareMathOperator*{\argmax}{arg\,max}
\DeclareMathOperator{\E}{\mathbb{E}}
\newcommand{\R}{\mathbb{R}}

\begin{document}

\title{Random Access Communication for Wireless Control Systems with Energy Harvesting Sensors}
\author{Miguel Calvo-Fullana, Carles Ant\'on-Haro, Javier Matamoros, and Alejandro Ribeiro
\thanks{
This work is supported by ARL DCIST CRA W911NF-17-2-0181 and the Intel Science and Technology Center for Wireless Autonomous Systems.
}
\thanks{
M. Calvo-Fullana, and A. Ribeiro are with the Department of Electrical and Systems Engineering, University of Pennsylvania, Philadelphia, PA 19104, USA (e-mail: \mbox{cfullana}@seas.upenn.edu; \mbox{aribeiro}@seas.upenn.edu).
}
\thanks{
C. Ant\'on-Haro, and J. Matamoros are with the Centre Tecnol\`ogic de Telecomunicacions de Catalunya (CTTC/CERCA), 08860 Castelldefels, Barcelona, Spain (e-mail: \mbox{carles.anton}@cttc.cat; \mbox{javier.matamoros}@cttc.cat).
}
\thanks{
This work has been presented in part at the 2017 American Control Conference (ACC)\cite{calvo2017random}.
}
}

\maketitle

\begin{abstract}
In this paper, we study wireless networked control systems in which the sensing devices are powered by energy harvesting. We consider a scenario with multiple plants, where the sensors communicate their measurements to their respective controllers over a shared wireless channel. Due to the shared nature of the medium, sensors transmitting simultaneously can lead to packet collisions. In order to deal with this, we propose the use of random access communication policies and, to this end, we translate the control performance requirements to successful packet reception probabilities. The optimal scheduling decision is to transmit with a certain probability, which is adaptive to plant, channel and battery conditions. Moreover, we provide a stochastic dual method to compute the optimal scheduling solution, which is decoupled across sensors, with only some of the dual variables needed to be shared between nodes. Furthermore, we also consider asynchronicity in the values of the variables across sensor nodes and provide theoretical guarantees on the stability of the control systems under the proposed random access mechanism. Finally, we provide extensive numerical results that corroborate our claims.
\end{abstract}

\begin{IEEEkeywords}
Energy harvesting, networked control systems, random access communication.
\end{IEEEkeywords}

\IEEEpeerreviewmaketitle

\section{Introduction}
\label{sec:Introduction}

The rapid pace of development of technologies such as robotic automation, smart homes, autonomous transportation,  and the internet of things is causing a dramatic increase in the average number of sensors in modern control systems. Usually, the previously mentioned technologies rely on networked control systems, and tend to incorporate wireless sensing devices to perform the monitoring of physical processes. These sensors might be deployed in large quantities and over large areas, making the replacement of their batteries a difficult and costly task. This has led to an increasing interest in alternative ways of powering wireless devices. An important technology that has recently emerged as capable of alleviating the limitations imposed by traditional battery operation is Energy Harvesting (EH). The use of energy harvesting technologies allows the devices to obtain energy from their environment (with common sources being solar, wind or kinetic energy \cite{vullers2010energy}). In turn, this removes some of the limitations imposed by traditional battery operation and grants an increase to the expected lifetime of the devices.

The study of communication systems powered by energy harvesting has recently received considerable attention. Current results available in the literature range from throughput maximization \cite{yang2012optimal,tutuncuoglu2012optimum,ozel2011transmission,ho2012optimal}, source-channel coding \cite{calvo2017reconstruction,castiglione2014energy,orhan2014source,castiglione2012energy}, estimation \cite{yang2013wireless,knorn2015distortion,calvo2016sensor}, and others (see \cite{ulukus2015energy} for a comprehensive overview). However, in general, limited attention has been given to the use of energy harvesting technologies in control applications. Most of the works currently available deal with the estimation of dynamical systems with sensors powered by energy harvesting \cite{nayyar2013optimal,nourian2014optimal,li2017power,huang2017event}. Nonetheless, the more explicit study of closed-loop system stability under energy harvesting constraints has been less studied, and only for single plant scenarios \cite{watkins2017battery,watkins2017stability}.

In this paper, we consider the multi-plant problem of scheduling communication between sensor nodes and their respective controllers. For control systems with classically powered sensor nodes (i.e., not energy harvesting), the scheduling problem in wireless networked control systems has been previously studied in several forms. The most common approach to this problem is the design of centralized scheduling policies. In such setup, in order to avoid packet collisions between the transmissions of the nodes, there exists an overseeing entity specifying which sensor is allowed to transmit at a given time slot. These type of policies might be static\cite{zhang2001stability,hristu2001feedback} or of a more dynamic nature, where centralized decisions can be taken based on plant state information \cite{walsh2002stability} or others. Decentralized policies have received less attention, with the authors in \cite{gatsis2016random} proposing a random access mechanism that adapts to channel conditions.

In our case, we study the scenario in which the sensors are powered by energy harvesting, and we focus on the design of decentralized scheduling policies. We consider the coexistence of multiple plants, with sensor nodes transmitting their measurements to their controllers over a shared wireless medium. Due to this, multiple sensors accessing the medium at the same time can cause collisions, leading to the unsuccessful reception of the sensor measurements by the controller. To mitigate this, we propose to use a random access communication scheme. First, we abstract the required control performance into a required successful packet reception probability. Under this abstraction, a Lyapunov function of each control loop is required to decrease at a given average rate. Then, we pose the random access mechanism as an stochastic optimization problem where the required successful reception probabilities act as constraints of the optimization problem. Then, the energy harvesting constraints are introduced into the problem in an average manner and and we modify the formulation to allow us to ensure time slot to time slot causality in the stochastic framework. To solve the optimization problem, we resort to a primal-dual stochastic subgradient method \cite{ribeiro2010ergodic}. At a given time slot, the resulting scheduling decision is to transmit with a certain probability, which is adaptive to the plant, channel and battery conditions. The resulting policy requires minimal coordination, with only some of the dual variables being shared between the sensors. Furthermore, we consider the possibility of asynchronism between the sensor nodes (i.e., nodes with outdated dual information) and provide theoretical guarantees that ensure the stability of all control loops under these conditions when using our proposed scheme. Finally, we validate our policy by means of simulations, which illustrate its ability to adapt to environmental conditions and satisfy the stability of all control loops. 

The rest of the paper is organized as follows. In Section \ref{sec:SystemModel} we introduce the system model and provide details on its control, communication and control performance aspects. Section \ref{sec:RandAccComm} develops the proposed random access communication scheme and we discuss how to adapt it to deal with energy harvesting. In Section \ref{sec:RandAccAlg} we introduce the algorithm used to obtain the random access communication policy. The stability of the system under the proposed policy is studied in Section \ref{sec:Stability}. After this, we devote Section \ref{sec:NumRes} to simulations assessing the performance of the proposed random access mechanism. Finally, we provide some concluding remarks in Section \ref{sec:Conclusions}.

\section{System Model}
\label{sec:SystemModel}

\begin{figure}[t]
    \centering
    \includegraphics[width=\columnwidth]{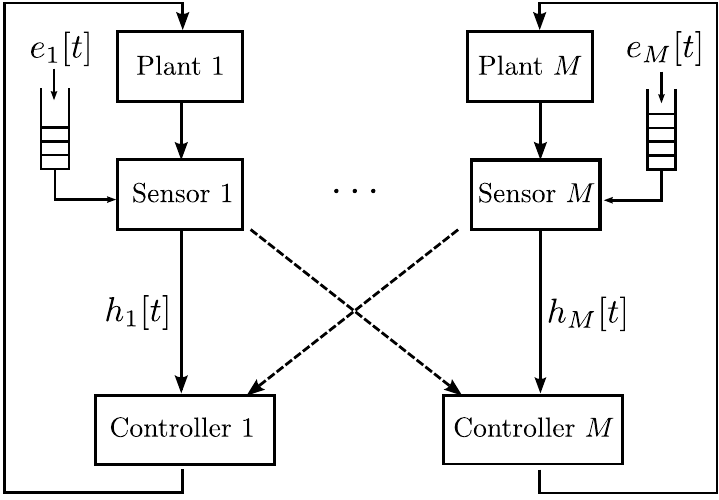} 
    \caption{System model.} 
    \label{fig:systemModel}
\end{figure}

Consider the system model shown in Figure \ref{fig:systemModel}. This scenario consists of $M$ different plants, which have their system state measured by sensor nodes powered by energy harvesting. The energy harvesting process imposes causality constraints on the transmission capabilities of the sensors, as sensors cannot transmit if they have not harvested sufficient energy. The measurements collected by the sensor nodes have then to be wirelessly transmitted to their respective controllers in order to ensure plant stability. However, the wireless medium over which the sensors transmit is shared. This implies that multiple sensors transmitting simultaneously can led to packet collisions, with the consequential lack of packet delivery. It is our objective to design transmission policies that adapt to the wireless medium and the energy harvesting process of the sensors, and are capable of stabilizing all control loops.

\subsection{Control Model}

We consider a group of $M$ plants and use $x_i[t]\in \R^{n_i}$ to denote the state of the $i$-th plant at time $t$. Plant dynamics are dictated by a linear time-invariant system in which plant control is contingent on the successful reception of information from the sensors. Define then the indicator variable $\gamma_i[t] \in \{0,1\}$ to signify with the value $\gamma_i[t]=1$ that the transmission of the $i$-th sensor at the $t$-th time slot has been successfully received by the $i$-th controller. If information is successfully received, we have $\gamma_i[t]=1$, in which case the controller closes the loop and the state evolves according to the closed loop dynamics described by the matrix $A_{c,i} \in \R^{n_i \times n_i}$. If, on the contrary, $\gamma_i[t]=0$, the state evolves in open loop as described by the matrix $A_{o,i} \in \R^{n_i \times n_i}$. We then have that the state $x_i[t]$ evolves according to the switched state linear dynamics

\begin{align}
x_i[t+1]=
\begin{cases}
    A_{c,i}x_i[t]+w_i[t],            &   \text{if } \gamma_i[t]=1, \\
    A_{o,i}x_i[t]+w_i[t],            &   \text{if } \gamma_i[t]=0,
\end{cases}
\label{eq:stateEvol}
\end{align}

where $w_i[t]$ correspond to independent and identically distributed (i.i.d.) Gaussian noise with covariance $C_i$. The design of the controllers is not the focus of this paper. The matrices are assumed given and are such that the closed loop matrix $A_{c,i}$ produces stable dynamics. The open loop matrix $A_{o,i}$ may produce stable or unstable dynamics but the problem is of most interest when the open loop dynamics are unstable.

\subsection{Communication Model}

In the control model we have defined the variables $\gamma_i[t]$ to signify the successful reception of the sensor measurements. This is a random variable whose distribution is dependent on the chosen communication policy. We consider a time-slotted communication model. At every time slot $t$, the $i$-th sensor node decides to transmit with probability $z_i[t] \in [0,1]$, where we denote $z_i[t]$ as the scheduling variable. Then, if multiple sensors transmit during the same time slot, we consider that a collision occurs with probability $q_c \in [0,1]$. If a collision occurs, then none of the colliding packets are received. Therefore, the probability of the $i$-th sensor node transmitting at time slot $t$ and not colliding with any other transmission is given by $z_i[t]\prod_{j\neq i}\left(1-q_c z_j[t]\right)$. Apart from collisions, packet loss can also occur due to incorrect decoding. The probability of successful decoding is dependent on the channel conditions at the $i$-th link during time slot $t$, which we denote by $h_i[t]$. Channel states are considered independent across the $M$ systems. Further, we consider a block fading model \cite{goldsmith2005wireless}, whereby the channel states $h_i[t]$ are i.i.d. over time slots and constant during a time slot. The probability of successfully decoding a packet given the channel state is denoted by $q(h_i[t])$, which is a continuous and strictly increasing function $q : \R^+ \to [0,1]$ (We show in Fig. \ref{fig:decodingFun} a typical decoding function). Also, for notational compactness, we also define $q_i[t] \triangleq q(h_i[t])$. Then, the probability of successful reception $\gamma_i[t]$ is given by
\begin{align}
\Pr\left(\gamma_i[t]=1\right)=q_i[t] z_i[t]\prod\limits_{j\neq i}\left(1-q_c z_j[t]\right).
\label{eq:probRX}
\end{align}
This expression simply corresponds to the successful decoding probability multiplied by the probability of transmitting without colliding. Also, we assume that sensor nodes have knowledge of their channel state before transmitting (In practice, this is usually achieved with pilot signals \cite{goldsmith2005wireless}).

\begin{figure}[t]
    \centering
    \includegraphics[width=\columnwidth]{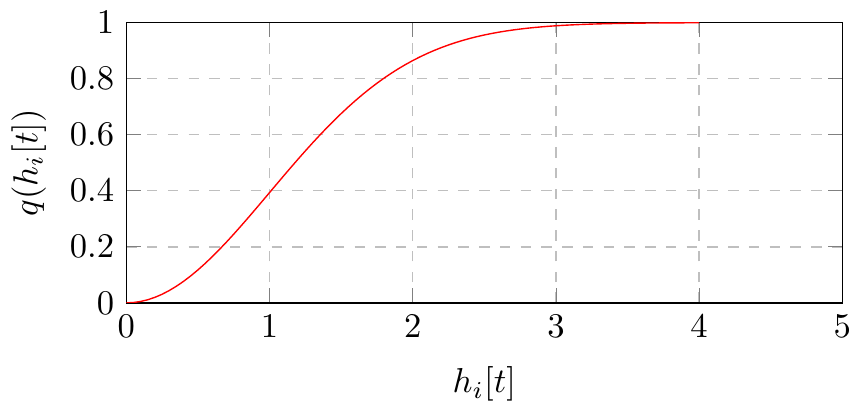} 
    \caption{Probability of decoding as a function of the channel state.}
    \label{fig:decodingFun}
\end{figure}

\subsection{Control Performance}

The control loop of each plant is closed with a probability given by equation \eqref{eq:probRX}. Since it is our objective to design communication policies that satisfy a desired control performance, we aim to establish a relationship between the control performance and the probability of successful reception. We can do so by the following proposition.
\begin{proposition}[Control performance abstraction \cite{gatsis2016random}]
\label{prop:ControlAbstraction}
Consider the switched system described by \eqref{eq:stateEvol} with $\gamma_i[t]$ given by a sequence of i.i.d. Bernoulli random variables, and the quadratic Lyapunov function $V_i(x_i)=x_i^TP_ix_i$, with $P_i \in \R^{n_i \times n_i}$ positive definite. Then the function $V_i(x_i)$ decreases at an average rate $\rho_i < 1$, denoted by 
\begin{align}
\E [V_i (x_i[t+1])|x_i[t]] \leq \rho_i V(x_i[t])+\tr(P_i C_i)
\label{eq:ControlPerfDecrease}
\end{align}
if and only if $\Pr(\gamma_i[t]=1)\geq p_i$, where $p_i$ is given by
\begin{align}
p_i=\min\limits_{\theta\geq 0} \left\{ \theta A_{c,i}^T P_i A_{c,i} +(1-\theta)A_{o,i}^T P_i A_{o,i} \leq \rho_i P_i \right\}
\label{eq:ControlPerfSDP}
\end{align}
\end{proposition}
\begin{proof}
By particularizing the function $V_i(x_i)=x_i^TP_ix_i$ with the system dynamics \eqref{eq:stateEvol}, we can write the equation 
\begin{align}
\E [V_i (x_i[t+1])|x_i[t]] &=
 x_i[t] A_{c,i}^T P_i A_{c,i} x_i[t] \Pr(\gamma_i[t]=1) \nonumber\\
&+ x_i[t] A_{o,i}^T P_i A_{o,i} x_i[t] \Pr(\gamma_i[t]=0) \nonumber\\
&+ \tr(P_i C_i).
\end{align}
Then, by substituting this expression in the left hand side of the average decrease inequality \eqref{eq:ControlPerfDecrease} we have the following inequality
\begin{align}
x_i[t] &A_{c,i}^T P_i A_{c,i} x_i[t] \Pr(\gamma_i[t]=1) \nonumber\\
+ & x_i[t] A_{o,i}^T P_i A_{o,i} x_i[t] \Pr(\gamma_i[t]=0) \leq \rho_i x_i[t] P_i x_i[t].
\end{align}
Since this condition needs to hold for all $x_i[t]$, we can equivalently rewrite this condition as the following linear matrix inequality
\begin{align}
A_{c,i}^T P_i A_{c,i} &\Pr(\gamma_i[t]=1) \nonumber\\
&+ A_{o,i}^T P_i A_{o,i} (1-\Pr(\gamma_i[t]=1)) \leq \rho_i P_i,
\end{align}
where we have also used the fact that $\Pr(\gamma_i[t]=0)=(1-\Pr(\gamma_i[t]=1))$. Then, the $\Pr(\gamma_i[t]=1)$ values satisfying this inequality define a convex set of which there is a minimum value $p_i$ such that the condition is equivalent to $\Pr(\gamma_i[t]=1)\geq p_i$.
\end{proof}

This proposition allows us to establish a connection between the control performance and the packet transmisisons. By solving the semidefinite program \eqref{eq:ControlPerfSDP}, we obtain the successful reception probabilities $p_i$ that allow us to satisfy the required control performance. Then, we simply need to design communication policies that satisfy $\Pr(\gamma_i[t]=1)\geq p_i$ for all systems.

\section{Random Access Communication}
\label{sec:RandAccComm}

We aim to design communication policies that satisfy the successful packet reception probabilities given by Proposition \ref{prop:ControlAbstraction}. Under an assumption of ergodic processes, the successful packet reception probabilities are given by the long term behavior of expression \eqref{eq:probRX}. Hence, in order to stabilize the control system to the required control performance, the scheduling variables $z_i[t]$ need to satisfy the following long term constraint
\begin{align}
p_i \leq \lim\limits_{t \to \infty} \frac{1}{t}\sum\limits_{l=1}^{t} q_i[l] z_i[l]\prod\limits_{j\neq i}\left(1-q_c z_j[l]\right)
\label{eq:ergProbTX}.
\end{align}
Since we are working under the assumption of ergodicity, we can write the previous limit as the expected value over channel realizations. That is,
\begin{align}
p_i \leq \E \biggl[ q_i z_i \prod\limits_{j\neq i}\left(1-q_c z_j\right) \biggr],
\end{align}
and, since scheduling decision are independent over nodes, we can further rewrite the previous expression as 
\begin{align}
p_i \leq \E \left[ q_i z_i \right] \prod\limits_{j\neq i}\left(1- \E \left[ q_c z_j \right] \right).
\end{align}
Aside from the stabilization of all control loops, we also want to minimize the number of times that a sensor node accesses the medium. We do this by the introduction of the objective function $\sum_{i=1}^{M} \E z_i^2 $. Then, we formulate the following optimization problem
\begin{subequations}
\begin{align}
   \underset{z_i \in \mathcal{Z}}{\text{minimize}}  \quad   & \sum_{i=1}^{M} \E z_i^2  \\
   \text{subject to}
		\quad   & p_i \leq \E \left[ q_i z_i \right] \prod\limits_{j\neq i}\left(1- \E \left[ q_c z_j \right]\right),  i=1,\ldots,M \label{eq:OptJointNoEHConst}
\end{align}
\label{eq:OptJointNoEHProblem}
\end{subequations}
where $\mathcal{Z}:= \{z_i : \R^+ \to [0,1]\}$ is the set of functions $\R^+ \to [0,1]$ taking values on $[0,1]$. Notice that, while scheduling decisions are statistically independent across sensors, solving the previous optimization problem requires it being done in a centralized manner (as constraint \eqref{eq:OptJointNoEHConst} is coupled across sensors). Nonetheless, we can separate the problem in a per sensor manner by taking the logarithm of constraint \eqref{eq:OptJointNoEHConst} as follows
\begin{subequations}
\begin{align}
   \underset{
   \begin{subarray}{l}
  		z_i \in \mathcal{Z}, \\
  		s_{ij} \in [0,1]
  	\end{subarray}   
   }
   {\text{minimize}}  \quad   & \sum_{i=1}^{M} \E z_i^2 \\
   \text{subject to}
		\quad   & \log(p_i) \leq \log(s_{ii}) + \sum\limits_{j\neq i}\log\left(1-s_{ij}\right), \nonumber\\
					&	\quad \quad \quad \quad \quad \quad i=1,\ldots,M\\
		\quad   & s_{ii} \leq \E q_i z_i,  \quad i=1,\ldots,M\\
		\quad   & s_{ij} \geq \E q_c z_j, \quad i=1,\ldots,M, j \neq i
\end{align}
\label{eq:OptSeparateNoEHProblem}
\end{subequations}
where we have introduced the auxiliary variables $s_{ii}$ and $s_{ij}$ and converted the logarithm of the product into a sum of logarithms. Solving the optimization problem \eqref{eq:OptSeparateNoEHProblem} is equivalent to solving \eqref{eq:OptJointNoEHProblem}. Under the assumption that this problem is strictly feasible, that is, that there exist schedules $z_i$ capable of satisfying $p_i < \E \bigl[ q_i z_i \prod_{j\neq i}\left(1-q_c z_j\right) \bigr]$, the goal is then to design an algorithm such that the instantaneous scheduling decisions $z_i[t]$ satisfy $\E \bigl[ z_i[t] \bigr]=z_i$.

\subsection{Random Access Communication with Energy Harvesting}

We have proposed a random access optimization problem that allows us to stabilize all control loops. However, the formulation previously introduced does not account for either the energy consumption nor the energy harvesting process. We consider that the $i$-th sensor at time slot $t$ acquires $e_i[t]$ units of energy and stores it in a battery of finite capacity $b_i^{\max}$. Further, we assume the energy harvesting process to be stationary with mean $\E \bigl[ e_i[t] \bigr]$. We consider that the sensor nodes consume one unit of energy per channel access, hence the scheduling variable also represents the power consumption of a medium access.
 Then, in order to ensure that the sensor nodes only use the energy available in their batteries, we have the following energy causality constraint
\begin{align}
z_i[t] \leq b_i[t].
\label{eq:energyCausality}
\end{align}
where $b_i[t]$ is the battery state of node $i$ at time $t$. Further, the battery of the nodes evolves according to the following dynamics
\begin{align}
b_i[t+1]=\biggl[b_i[t]-z_i[t]+e_i[t]\biggr]_{0}^{b_i^{\max}},
\label{eq:batteryEvol}
\end{align}
where $[ \cdot ]_{0}^{b_i^{\max}}$ denotes the projection to the interval $[0,b_i^{\max}]$. However, these constraints are coupled across time slots and cannot be directly introduced into the stochastic optimization problem \eqref{eq:OptSeparateNoEHProblem}. In order to circumvent this, we consider the long-term behavior of the energy causality constraints \eqref{eq:energyCausality}, which, by recursively substituting the battery dynamics \eqref{eq:batteryEvol} in \eqref{eq:energyCausality} can be written as follows
\begin{align}
\lim\limits_{t \to \infty} \frac{1}{t} \sum\limits_{l=1}^{t} z_i[l] \leq \lim\limits_{t \to \infty} \frac{1}{t} \sum\limits_{l=1}^{t} e_i[l].
\label{eq:ergEnergyCausality}
\end{align}
That is, in the long term, the battery state is dominated by the harvested energy. Then, due to the ergodicity of the scheduling variables $z_i[t]$ and the energy harvesting process $e_i[t]$, the previous expression \eqref{eq:ergEnergyCausality} can be simply written as the expectation with respect to the channel states $h_i[t]$ and the energy harvesting process $e_i[t]$, as follows
\begin{align}
 \E \bigl[ z_i \bigr] \leq \E \bigl[ e_i \bigr].
 \label{eq:ergEnergyCausalityExpectation}
\end{align}
This constraint simply implies that, on average, the energy spent for transmitting has to be lower than the harvested energy. Then, by introducing constraint \eqref{eq:ergEnergyCausalityExpectation} into the optimization problem \eqref{eq:OptSeparateNoEHProblem} we have the following problem
\begin{subequations}
\begin{align}
   \underset{
   \begin{subarray}{l}
  		z_i \in \mathcal{Z}, \\
  		s_{ij} \in [0,1]
  	\end{subarray}   
   }
   {\text{minimize}}  \quad   & \sum_{i=1}^{M} \E z_i^2 \\
   \text{subject to}
		\quad   & \log(p_i) \leq \log(s_{ii}) + \sum\limits_{j\neq i}\log\left(1-s_{ij}\right), \nonumber\\
					&	\quad \quad \quad \quad \quad \quad i=1,\ldots,M \label{eq:optProbNoAuxC0}	 \\
		\quad   & s_{ii} \leq \E q_i z_i,  \quad i=1,\ldots,M \label{eq:optProbNoAuxC1}	 \\
		\quad   & s_{ij} \geq \E q_c z_j, \quad i=1,\ldots,M, j \neq i \label{eq:optProbNoAuxC2}	\\
		\quad   & \E z_i \leq \E e_i 	 \quad i=1,\ldots,M	 \label{eq:optProbNoAuxC3}	
\end{align}
\label{eq:optProbNoAux}
\end{subequations}
However, substituting the time slot to time slot constraints \eqref{eq:energyCausality} by the average ones \eqref{eq:ergEnergyCausalityExpectation} does not ensure that they are satisfied at each time slot. This means that solutions to the optimization problem \eqref{eq:optProbNoAux} do not necessarily satisfy the energy causality constraints $z_i[t] \leq b_i[t]$. To overcome this problem, and ensure causality, we introduce the following modified problem formulation
\begin{subequations}
\begin{align}
   \underset{
   \begin{subarray}{l}
  		z_i \in \mathcal{Z}, \\
  		s_{ij} \in [0,1], \\
  		y_{ij} \in [0,\bar{y}_{ij}]
  	\end{subarray}   
   }
   {\text{minimize}}  \quad   & \sum_{i=1}^{M} \E z_i^2  + \sum_{i=1}^{M} \sum_{j=1}^{M} \E \bar{\nu}_{ij} y_{ij} \\
   \text{subject to}
		\quad   & \log(p_i) \leq \log(s_{ii}) + \sum\limits_{j\neq i}\log\left(1-s_{ij}\right), \nonumber\\
					&	\quad \quad \quad \quad \quad \quad i=1,\ldots,M \label{eq:optProbAuxC0}\\
		\quad   & s_{ii} \leq \E q_i z_i + y_{ii},  \quad i=1,\ldots,M \label{eq:optProbAuxC1}\\
		\quad   & s_{ij} \geq \E q_c z_j - y_{ij}, \quad i=1,\ldots,M, j \neq i \label{eq:optProbAuxC2}\\	
		\quad   & \E z_i \leq \E e_i 	 \quad i=1,\ldots,M \label{eq:optProbAuxC3}
\end{align}
\label{eq:optProbAux}
\end{subequations}
This optimization problem has been modified by the introduction of the auxiliary variables $y_{ij}$ in constraints \eqref{eq:optProbAuxC1} and \eqref{eq:optProbAuxC2}, as well as in the objective function. The auxiliary variable $y_{ij}$ is forced to take values in the interval $y_{ij} \in [0,\bar{y}_{ij}]$, where $\bar{y}_{ij}$ is a system-dependent constant. The term $\sum_{i=1}^{M} \sum_{j=1}^{M} \E \bar{\nu}_{ij} y_{ij}$ in the objective function has the constant $\bar{\nu}_{ij}$, where the $\bar{\nu}_{ij}$ value is an upper bound on the Lagrange multipliers of constraints \eqref{eq:optProbAuxC1} and \eqref{eq:optProbAuxC2}. This modified problem formulation allows us to ensure that even though the energy constraint \eqref{eq:optProbAuxC3} is on average form, the energy causality constraints are satisfied in a time slot to time slot basis, as we will show in the upcoming sections.

\section{Random Access Algorithm}
\label{sec:RandAccAlg}

In this section, we aim to solve the optimization problem \eqref{eq:optProbAux}. For notational compactness, let us define the vector $z=\{z_i,s_{ij},y_{ij}\}$ collecting all the primal variables and the vector $\lambda=\{\phi_i,\nu_{ij},\beta_i\}$ collecting the dual variables. Further, we collect the implicit primal variable constraints in the set $\mathcal{X} \triangleq \{ z_i \in \mathcal{Z}, s_{ij} \in [0,1], y_{ij} \in [0,\bar{y}_{ij}] \}$. Then, the Lagrangian of problem \eqref{eq:optProbAux} can be written as
\begin{align}
\mathcal{L}(z,\lambda)&= \sum_{i=1}^{M} \E z_i^2 + \sum_{i=1}^{M} \sum_{j=1}^{M} \E \bar{\nu}_{ij} y_{ij} \nonumber\\
&+ \sum\limits_{i=1}^{M} \phi_i\left(\log\left(p_i\right) - \log\left(s_{ii}\right) - \sum\limits_{j\neq i}\log\left(1-s_{ij}\right) \right) \nonumber\\
&+ \sum\limits_{i=1}^{M} \nu_{ii}\left(s_{ii} - \E q_i z_i - y_{ii}\right)  \nonumber\\
&+ \sum\limits_{i=1}^{M} \sum_{j \neq i} \nu_{ij}\left(\E q_c z_j - y_{ij} - s_{ij}\right) \nonumber\\
&+ \sum\limits_{i=1}^{M}\beta_i\left(\E z_i - \E e_i\right).
 \label{eq:optProbAuxLagrangian}
\end{align}
The Lagrange dual function of this problem is given by
\begin{align}
g(\lambda)=\min_{z \in \mathcal{X}} \mathcal{L}(z,\lambda).
\end{align}
Note that, while the primal problem is infinite dimensional, the dual problem has a finite number of variables (the dual variables). Furthermore, for this problem, the duality gap can be shown to be zero \cite{ribeiro2012optimal}. Hence, we resort to a dual subgradient method to solve the optimization problem. However, the sensor nodes have no knowledge of the probability distribution over which the expectation is taken. In order to overcome this, we substitute the random variables by their instantaneous values, which are known by the sensors. Finally, by reordering the Lagrangian \eqref{eq:optProbAuxLagrangian}, the scheduling variables $z_i[t]$ are given by the following minimization
\begin{align}
z_i[t]:=\argmin\limits_{z_i\in [0,1]} z_i \left( z_i-\nu_{ii}[t]q_i[t]+q_c \sum_{j \neq i} \nu_{ji}[t]+\beta_i[t] \right),
\end{align}
which is separated across sensors and leads to the following closed form solution
\begin{align}
z_i[t]:= \frac{1}{2}\biggl[ \nu_{ii}[t]q_i[t]-q_c \sum_{j \neq i} \nu_{ji}[t]-\beta_i[t] \biggr]_0^1.
 \label{eq:primalZClosed}
\end{align}
The resulting optimal scheduling policy is to transmit at time slot $t$ with the probability given by \eqref{eq:primalZClosed}. This is a policy that dynamically adapts to the time-varying conditions of the system. Namely, the dual variables $\nu_{ii}[t]$ and $\nu_{ji}[t]$ depend on the stability of all the plants, the $q_i[t]$ variable is dependent on the channel state $h_i[t]$, and the dual variables $\beta_i[t]$ depend on the energy harvesting process $e_i[t]$. The rest of the primal variables $s_{ii}[t]$ and $s_{ij}[t]$ can be found by the minimization
\begin{align}
s_{ii}[t]:=\argmin\limits_{s_{ii}\in [0,1]} -\phi_i[t] \log(s_{ii})+ \nu_{ii}[t] s_{ii},
\end{align}
\begin{align}
s_{ij}[t]:=\argmin\limits_{s_{ij}\in [0,1]} - \phi_i[t] \log(1-s_{ij}) - \nu_{ij}[t] s_{ij},
\end{align}
which, again, are separated across sensors and have the following closed forms solutions 
\begin{align}
s_{ii}[t]:=\left[\frac{\phi_i[t]}{\nu_{ii}[t]}\right]_0^1,
\quad \quad
s_{ij}[t]:=\left[1-\frac{\phi_i[t]}{ \nu_{ij}[t]}\right]_0^1.
 \label{eq:primalSClosed}
\end{align}
In a similar way, the auxiliary variables $y_{ij} [t]$ can be computed as the solution of the minimization
 \begin{align}
y_{ij} [t] := \argmin_{y_{ij} \in [0,\bar{y}_{ij}]} y_{ij} \left( \bar{\nu}_{ij} - \nu_{ij}[t]\right),
 \label{eq:auxVarYMin}
\end{align}
which is a thresholding condition. The auxiliary variable $y_{ij} [t]$ takes the value $y_{ij} [t]:=0$ if $\nu_{ij} [t] \leq \bar{\nu}_{ij}$ and $y_{ij} [t] := \bar{y}_{ij}$ if $\nu_{ij} [t] > \bar{\nu}_{ij}$. Next, since the dual function is concave, we can perform a subgradient ascent on the dual domain. The corresponding dual variable updates are given by
\begin{align}
\phi_i[t+1]:=\biggl[\phi_i[t]+\epsilon \biggl(\log &  \left(  p_i\right) -\log\left(s_{ii}[t]\right)  \nonumber\\
&- \sum\limits_{j\neq i}\log\left(1-s_{ij}[t]\right) \biggr) \biggr]^+
\end{align}
\begin{align}
\nu_{ii}[t+1]:=\left[\nu_{ii}[t]+\epsilon \biggl( s_{ii}[t]-z_i[t]q_i[t] - y_{ii}[t]  \biggr) \right]^+
 \label{eq:dualUptNuii}
\end{align}
\begin{align}
\nu_{ij}[t+1]:=\left[\nu_{ij}[t]+\epsilon \biggl( q_c z_i[t]-s_{ij}[t] - y_{ij}[t] \biggr) \right]^+
 \label{eq:dualUptNuij}
\end{align}
\begin{align}
\beta_i[t+1]:=\left[\beta_i[t]+ \epsilon \biggl( z_i[t]-e_i[t] \biggr) \right]^+
\end{align}
where, in order to have an algorithm than can be run in an online manner, we have considered a fixed step sized $\epsilon$. For notational compactness, we also write the dual update in a concatenated vector form as $\lambda[t+1]:=\bigl[\lambda[t]+\epsilon s[t] \bigr]^+$, where $s_i[t]$ corresponds to the stochastic subgradient. The steps in the resulting random access mechanism are summarized in Algorithm \ref{alg:Algorithm}.

Also, it is important to note that we can establish a parallel relationship between the dual variables $\beta_i[t]$ associated to the energy constraint  $\E z_i \leq \E e_i$  and the actual battery state $b_i[t]$. This relationship is given by the expression $\beta_i[t]=\epsilon  \bigl( b_i^{\max} - b_i[t] \bigr)$. Hence, a mirrored symmetry (scaled by the step size $\epsilon$) exists between these variables. This relationship will be crucial in ensuring the energy causality of the algorithm, as we show next.

\begin{algorithm}[t]
    \caption{Random access scheduling algorithm.}
    \label{alg:Algorithm}
    \begin{algorithmic}[1]
        \State \textbf{Initialize:} Initialize the dual variables to $\phi_i[0]:=0$, $\nu_{ij}[0]:=0$, and $\beta_i[0]:=\epsilon \bigl( b_i^{\max}-b_i[0] \bigr)$.
        \State \textbf{Step 1:} Medium access decision
        \State $z_i[t]:= \frac{1}{2}\biggl[ \nu_{ii}[t]q_i[t]-q_c \sum_{j \neq i} \nu_{ji}[t]-\beta_i[t] \biggr]_0^1$
        \State \textbf{Step 2:} Other primal variables 
        \State $s_{ii}[t]:=\left[\frac{\phi_i[t]}{\nu_{ii}[t]}\right]_0^1 $ and
         		  $s_{ij}[t]:=\left[1-\frac{\phi_i[t]}{\nu_{ij}[t]}\right]_0^1$
        \State \textbf{Step 3:} Auxiliary variable
        \State $y_{ij} [t] := \argmax\limits_{y_{ij} \in [0,\bar{y}_{ij}]} y_{ij} \left( \bar{\nu}_{ij} - \nu_{ij}[t]\right)$
        \State \textbf{Step 4:} The sensor updates the dual variables
        \State $\phi_i[t+1]:=\biggl[\phi_i[t]+\epsilon \big( \log\left(p_i\right) - \log\left(s_{ii}[t]\right)$
        		  
        		  \vspace{-2ex}
        		  \hspace{25ex}
        		  $-\sum\limits_{j\neq i}\log\left(1-s_{ij}[t]\right)\big)\biggr]^+$
        \State $\nu_{ii}[t+1]:=\bigg[\nu_{ii}[t]+\epsilon\big(s_{ii}[t]-z_i[t]q_i[t] - y_{ii}[t] \big)\bigg]^+$
        \State $\nu_{ij}[t+1]:=\bigg[\nu_{ij}[t]+\epsilon\big( q_c z_i[t]-s_{ij}[t] - y_{ij}[t] \big)\bigg]^+$
        \State $\beta_i[t+1]:=\bigg[\beta_i[t]+\epsilon\big( z_i[t]-e_i[t]\big)\bigg]^+$
        \State \textbf{Step 5:} Set $t:=t+1$ and go to Step 1.
    \end{algorithmic}
\end{algorithm}

\subsection{Energy Causality}
Now, we turn our attention to the study of the conditions required to satisfy the causality constraints, i.e., $z_i[t] \leq b_i[t]$ for all time slots. We have introduced the modified problem formulation \eqref{eq:optProbAux} to help achieve this. First, we show that this problem formulation allows us to upper bound the dual variables $\nu_{ij}[t]$ over all time slots.
\begin{proposition}
\label{prop:DualUpperBound}
Let the upper bound $\bar{y}_{ij}$ of the auxiliary variables $y_{ij}$ satisfy the inequality $\bar{y}_{ij} \geq \frac{1}{\epsilon}\left(\bar{\nu}_{ij}+2\epsilon\right)$. Then, the dual variables $\nu_{ij}[t]$ are upper bounded by $\nu_{ij}[t] \leq \bar{\nu}_{ij} + \epsilon$ for all time slots $t$.
\end{proposition}
\begin{proof}
The dual variable $\nu_{ij}[t]$ is updated according to the following equations
\begin{align}
\nu_{ii}[t+1]&:=\left[\nu_{ii}[t]+\epsilon \biggl( s_{ii}[t]-z_i[t]q_i[t] - y_{ii}[t]  \biggr) \right]^+ \\
\nu_{ij}[t+1]&:=\left[\nu_{ij}[t]+\epsilon \biggl( q_c z_i[t]-s_{ij}[t] - y_{ij}[t] \biggr) \right]^+,
 \label{eq:dualUptNuij}
\end{align}
 where the subgradient terms are upper bounded by $1$, namely $s_{ii}[t]-z_i[t]q_i[t] - y_{ii}[t] \leq 1$ and $q_c z_i[t]-s_{ij}[t] - y_{ij}[t] \leq 1$ for all time slots $t$. Hence, we have that the maximum increase of these dual variables in a given time slot is $| \nu_{ij}[t+1] - \nu_{ij}[t] | \leq \epsilon$ for all $i,j$ and $t$. Overall, the maximum value that the dual variables $\nu_{ij}[t]$ can take is controlled by the $y_{ij}[t]$ term. As long as $y_{ij}[t]=0$ the dual variables can increase in value, until $\nu_{ij}[t] = \bar{\nu}_{ij} + \epsilon$ and the auxiliary variable condition in \eqref{eq:auxVarYMin} is triggered, leading to the $y_{ij}[t]$ term taking the value $y_{ij} [t] = \bar{y}_{ij}$. Then, the next update of the dual variable is given by
 \begin{align}
 \nu_{ij}[t+1] & \leq \bigl[\bar{\nu}_{ij} +\epsilon - \epsilon y_{ij}[t]\bigr]^+ \nonumber\\
					& \leq \left[\bar{\nu}_{ij} + \epsilon +\epsilon - \epsilon \left( \frac{1}{\epsilon}\left(\bar{\nu}_{ij}+2\epsilon\right) \right)\right]^+ \nonumber\\
					& = 0.
\end{align}
 Since after this event, the dual variables take the zero value, the dual variables $\nu_{ij}[t]$ are necessarily upper bounded by $\nu_{ij}[t] \leq \bar{\nu}_{ij} + \epsilon$ for all time slots $t$.
\end{proof}

This proposition states that by ensuring the correct value of the parameter $\bar{y}_{ij}$ (which we can select freely), an upper bound on $\nu_{ij}[t]$ can be established. Then, by further appropriately selecting the battery size $b_i^{\max}$ of the nodes, we can ensure that energy use is causal to the energy harvested.

\begin{proposition}[Energy Causality]
\label{prop:EnergyCausality}
Let the battery capacity of the $i$-th sensor satisfy $b_i^{\max} \geq \frac{1}{\epsilon} \bar{\nu}_{ii} + 1$ and let $\bar{y}_{ij} \geq \frac{1}{\epsilon}\left(\bar{\nu}_{ij}+2\epsilon\right)$. Then, Algorithm \ref{alg:Algorithm} satisfies the energy consumption causality constraints $z_i[t] \leq b_i[t]$ for all time slots.
\end{proposition}
\begin{proof}
In order to satisfy the energy causality constraints $z_i[t] \leq b_i[t]$, it suffices to verify that no transmission occurs when there is no energy left in the battery. This implies that the scheduling variable $z_i[t]$ has to take the zero value when the battery $b_i[t]$ is empty. By equation \eqref{eq:primalZClosed}, it suffices to satisfy $\nu_{ii}[t]q_i[t]-q_c \sum_{j \neq i} \nu_{ji}[t]-\beta_i[t] \leq 0$, when $b_i[t]=0$. Note that the battery state $b_i[t]$ and the battery multipliers $\beta_i[t]$ are related by the expression $\beta_i[t]=\epsilon  \bigl( b_i^{\max} - b_i[t] \bigr)$. Hence, the battery being empty, $b_i[t]=0$, implies the battery multipliers taking the value $\beta_i[t]=\epsilon b_i^{\max}$. Therefore, the condition to be satisfied is  $\nu_{ii}[t]q_i[t]-q_c \sum_{j \neq i} \nu_{ji}[t]- \epsilon b_i^{\max} \leq 0$. Since $q_i[t] \leq 0$ and $q_c \geq 0$, we can further rewrite this inequality as $\nu_{ii}[t]- \epsilon b_i^{\max} \leq 0$. Then, by Proposition \ref{prop:DualUpperBound}, we have the upper bound on the dual variables $\nu_{ij}[t] \leq \bar{\nu}_{ij} + \epsilon$. This allows us to further rewrite the inequality as $\bar{\nu}_{ij} + \epsilon - \epsilon b_i^{\max} \leq 0$. Then, the battery size $b_i^{\max} \geq \frac{1}{\epsilon} \bar{\nu}_{ii} + 1$, ensures this inequality, and hence, that the energy constraints $z_i[t] \leq b_i[t]$ are satisfied for all time slots.
\end{proof}

According to this proposition, by choosing a sufficiently large battery size $b_i^{\max}$ we can make the energy consumption causal to the energy harvesting process. This is due to the modified problem formulation proposed in \eqref{eq:optProbAux}. In the original problem \eqref{eq:optProbNoAux}, a dual ascent algorithm can lead to the dual variables becoming arbitrarily large. This is not the case when introducing the auxiliary variables $y_{ij}$, as shown by Proposition \ref{prop:DualUpperBound}. Then, by Proposition \ref{prop:EnergyCausality}, the bound on the dual variables allows us to establish conditions on the battery size that ensure energy causality.

\subsection{Asynchronous Operation}

In order for Algorithm \ref{alg:Algorithm} to properly function, the $i$-th sensor node requires the dual variables $\nu_{ij}[t]$ for $j \neq i$ of the other nodes. This is needed in order to compute the optimal scheduling variable $z_i[t]$ (cf. equation \eqref{eq:primalZClosed}). In order to ensure the robustness of our algorithm, we take into account the notion of asynchronicity in the data shared across the sensor nodes. This is to say that we consider the possibility of different nodes having different (out of date) values of the dual variables shared by the other nodes. Since nodes are powered by energy harvesting, this might happen when a node is unable to transmit or receive data due to lack of energy. Also, the consideration of asynchronicity includes the practical case in which the sensor nodes simply attach the value of their dual variable to the packet containing the measurement. Therefore, only sharing their dual variable when they need to transmit a measurement to their controller. To consider this, we introduce the asynchronicity model of \cite{bertsekas1989parallel} into our analysis. 

Let us define the set $T^i \subseteq \mathbb{Z}^+$ of all time slots in which the $i$-th node is capable of receiving and sending information. Then, we define a function $\pi^i[t]$, that for a given node and time slot, returns the most recent time slot at which the node was available. Namely, 
\begin{align}
\pi^i[t]:=\max \left\{\hat{t} \mid \hat{t} < t, \hat{t} \in T^i \right\}.
\end{align}
In a similar manner, we then define the function $\pi^i_j[t]:=\pi^j \bigl[ \pi^i[t] \bigr]$ to denote the most recent time slot the $i$-th node has received information sent by the $j$-th node. Then, at time slot $t$, the $i$-th node has knowledge of a possibly outdated vector of dual variables $\nu_{ij}$, given by
\begin{align}
\tilde{\nu}_i[t]=\left(\nu_{i1}[\pi^i_1[t]],\ldots,\nu_{iM}[\pi^i_M[t]]\right).
\end{align}
Further, we will denote by $\tilde{\lambda}$ the vector formed by the collection of the outdated duals together with the rest of the dual variables. Following, we can write the asynchronism into the dual variable update by defining the asynchronous stochastic subgradient $\tilde{s}_{\nu_{ij}}[t]$, corresponding to the $\nu_{ij}$ variable. We do so as follows
\begin{align}
\tilde{s}_{\nu_{ij}}[t]=
\begin{cases}
    s_{\nu_{ij}}[t],            &   \text{if } t \in T^j, \\
    0,     				&   \text{otherwise}.
\end{cases}
\end{align}
Simply meaning that, if $t \in T^j$, the ascent direction given by the subgradient $s_{\nu_{ij}}[t],$ is available. Otherwise, the dual variable is not updated. Then, we can concatenate all the subgradients of all dual variables into an asynchronous stochastic subgradient vector $\tilde{s}[t]$. Afterwards, the dual variable update is simply given by the usual expression but with the asynchronous stochastic subgradient, i.e., $\lambda[t+1]:=\left[\lambda[t]+\epsilon \tilde{s}[t] \right]^+$.

\section{Stability Analysis}
\label{sec:Stability}

In this section, we analyze the stability of the systems when operating under the proposed random access communication scheme. In order to do this, we leverage on the fact that the proposed scheme is a stochastic subgradient algorithm. Hence, we rely on duality theory arguments to show that the iterates generated by  Algorithm \ref{alg:Algorithm} satisfy the constraints of the optimization problem \eqref{eq:optProbAux} almost surely. Then, we further show that if the constants $\bar{\nu}_{ij}$ are chosen to upper bound the optimal Lagrange multipliers  $\nu^{\star}_{ij}$, then the iterates generated by Algorithm \ref{alg:Algorithm} also satisfy the constraints of the optimization problem \eqref{eq:optProbNoAux} (i.e., without the auxiliary variables). In turn, this guarantees by Proposition \ref{prop:ControlAbstraction} the stability of all control loops.

First, in order to ensure the convergence of Algorithm \ref{alg:Algorithm}, we need to assume an upper bound on the asynchronicity between the sensor nodes.
 
\begin{assumption}
\label{assump:asynchronicity}
There exists an upper bound $0 < B < \infty$ to the asynchronicity between nodes, such that for all time $t$ and nodes $i,j$ we have
\begin{align}
\max \left\{0,t-B+1\right\} \leq \pi^i_j[t] \leq t.
\end{align}	
\end{assumption}
This assumption simply implies that nodes are at most $B$ time slots out of synchronism and it is required to ensure the convergence of the variables. Now, we proceed to show the convergence of Algorithm \ref{alg:Algorithm}. We start by recalling a common property of the subgradient method.
\begin{proposition}
\label{prop:avgSubgradient}
Given dual variables $\lambda[t]$, the conditional expected value $\E \big[ s[t] | \lambda [t] \big]$ of the stochastic subgradient $s[t]$ is a subgradient of the dual function. Namely, for any $\lambda$, 
\begin{align}
\E \big[ s^T[t] | \lambda [t] \big] 	\big( \lambda [t]  - \lambda \big) \leq g(\lambda [t] ) - g(\lambda).
\end{align}
\end{proposition}
\begin{proof}
We intend to show that the expected value of the stochastic subgradient $s[t]$ given $\lambda[t]$ is a subgradient of the dual function $g(\lambda)$. To do this, we take the Lagrangian \eqref{eq:optProbAuxLagrangian} of optimization problem \eqref{eq:optProbAux}, given by
\begin{align}
\mathcal{L}&(z,\lambda)= \sum\displaystyle_{i=1}^{M} \E z_i^2 + \sum\displaystyle_{i=1}^{M} \sum\displaystyle_{j=1}^{M} \E \bar{\nu}_{ij} y_{ij} \nonumber\\
&+ \sum\displaystyle_{i=1}^{M} \phi_i\bigl(\log\left(p_i\right) - \log\left(s_{ii}\right) - \sum\displaystyle_{j\neq i}\log\left(1-s_{ij}\right) \bigr) \nonumber\\
&+ \sum\displaystyle_{i=1}^{M} \nu_{ii}\bigl(s_{ii} - \E q_i z_i - y_{ii}\bigr)  \nonumber\\
&+ \sum\displaystyle_{i=1}^{M} \sum\displaystyle_{j \neq i} \nu_{ij}\bigl(\E q_c z_j - y_{ij} - s_{ij}\bigr) \nonumber\\
&+ \sum\displaystyle_{i=1}^{M}\beta_i\bigl(\E z_i - \E e_i\bigr).
\end{align}
Then, take the dual function at time $t$, denoted by $g(\lambda[t])$ and remember that the dual function is given by $g(\lambda)=\min_{z \in \mathcal{X}} \mathcal{L}(z,\lambda)$. The primal variables that minimize this dual function are obtained by the primal minimization of Algorithm \ref{alg:Algorithm}, namely, $z_i[t]$, $s_{ij}[t]$ and $y_{ij}[t]$, given by equations \eqref{eq:primalZClosed}, \eqref{eq:primalSClosed}, and \eqref{eq:auxVarYMin}, respectively. Then, we write the dual function at time $t$ as
\begin{align}
&g(\lambda[t])= \sum\displaystyle_{i=1}^{M} \E z_i^2[t] + \sum\displaystyle_{i=1}^{M} \sum\displaystyle_{j=1}^{M} \E \bar{\nu}_{ij} y_{ij}[t] + \nonumber\\
&\sum\displaystyle_{i=1}^{M} \phi_i[t]\bigl(\log\left(p_i\right) - \log\left(s_{ii}[t]\right) - \sum\displaystyle_{j\neq i}\log\left(1-s_{ij}[t]\right) \bigr) \nonumber\\
&+ \sum\displaystyle_{i=1}^{M} \nu_{ii}[t] \E \bigl[s_{ii}[t] - q_i[t] z_i[t] - y_{ii}[t]\bigr] \nonumber\\
&+ \sum\displaystyle_{i=1}^{M} \sum\displaystyle_{j \neq i} \nu_{ij}[t] \E \bigl[ q_c z_j[t] - y_{ij}[t] - s_{ij}[t]\bigr] \nonumber\\
&+ \sum\displaystyle_{i=1}^{M}\beta_i[t] \E \bigl[z_i[t] - e_i[t] \bigr],
\end{align}
 where, due to its linearity, we have moved the expectation $\E[\cdot]$ out of the subgradients. Now, by compacting the Lagrange multipliers into a vector $\lambda[t]$ and the subgradients to $s[t]$, we can rewrite the dual function at time $t$ as
\begin{align}
g(\lambda[t]) &= \sum\displaystyle_{i=1}^{M} \E z_i^2[t] + \sum\displaystyle_{i=1}^{M} \sum\displaystyle_{j=1}^{M} \E \bar{\nu}_{ij} y_{ij}[t] \nonumber\\
& + \E \big[ s^T[t] | \lambda[t] \big] \lambda[t].
\label{eq:avgSubgradientDualLambdaT}
\end{align}
Further, for any arbitrary $\lambda$, the dual function $g(\lambda)$ can be bounded as
\begin{align}
g(\lambda) &\leq \sum\displaystyle_{i=1}^{M} \E z_i^2[t] + \sum\displaystyle_{i=1}^{M} \sum\displaystyle_{j=1}^{M} \E \bar{\nu}_{ij} y_{ij}[t] \nonumber\\
& + \E \big[ s^T[t] | \lambda[t] \big] \lambda.
\label{eq:avgSubgradientDualLambda}
\end{align}
Then, by substracting expression \eqref{eq:avgSubgradientDualLambda} from \eqref{eq:avgSubgradientDualLambdaT} we obtain
\begin{align}
\E \big[ s^T[t] | \lambda [t] \big] 	\big( \lambda [t]  - \lambda \big) \leq g(\lambda [t] ) - g(\lambda),
\end{align}
which is the desired inequality.
\end{proof}
The previous proposition states that, on average, the stochastic subgradient $s[t]$ is an ascent direction of the dual function $g(\lambda [t] )$. Then, the next step is to quantify the average reduction in distance to the optimal dual variables $\lambda^\star$ that occurs in a dual variable update step.
\begin{lemma}
\label{lmm:AverageDualDist}
Let $\E\bigl[\|s[t] \|^2| \lambda[t]\bigr] \leq S^2$  be a bound on the second moment of the norm of the stochastic subgradients $s[t]$. The dual updates of Algorithm \ref{alg:Algorithm}, satisfy the following inequality
\begin{align}
\label{eq:LmmAverageDualDist}
\E \big[ \|  \lambda^\star  - & \lambda[t+1]\big\|^2 | \lambda[t] \big]   \leq
\bigl( 1- \epsilon m+  2\epsilon^2 L B \bigr)  \big\|  \lambda^\star - \lambda[t] \big\|^2 \nonumber\\
&+\epsilon^2 S^2 + 2\epsilon^2 L B S^2 
- \epsilon \big( g(\lambda^\star) - g(\lambda[t]) \big),
\end{align}
where the constant $L>0$ corresponds to the $L$-Lipschitz continuity of the gradients of the dual function $g(\lambda)$ and $m>0$ to the strong concavity constant of the dual function $g(\lambda)$.
\end{lemma}
\begin{proof}
Let us consider the squared distance between the dual iterates $\lambda$ at time $t+1$ and their optimal value, i.e., $\big\|\lambda^\star-\lambda[t+1]\big\|^2$. By means of the dual update $\lambda[t+1]=\left[\lambda[t]+\epsilon \tilde{s}[t] \right]^+$, we can rewrite this expression as
\begin{align}
\big\|\lambda^\star-\lambda[t+1]\big\|^2& =\big\|\lambda^\star-\big[\lambda[t]+\epsilon\tilde{s}[t]\big]^+\big\|^2 \nonumber \\
&\leq \big\|\lambda^\star-\lambda[t]-\epsilon\tilde{s}[t]\big\|^2
\end{align}
where we have further upper bounded the expression by the nonexpansive property of the nonnegative projection. Then, we expand the square norm, yielding the expressions
\begin{align}
\big\|\lambda^\star-\lambda[t+1]\big\|^2 \leq  \big\|\lambda^\star-\lambda[t] & \big\|^2+\epsilon^2 \big\|\tilde{s}[t]\big\|^2 \nonumber\\
&- 2\epsilon \tilde{s}^T[t]	\big( \lambda^\star  - \lambda[t] \big).
\label{eq:LemmaDualGapExpansion}
\end{align}
We can further rewrite this inequality by adding and subtracting the term $2\epsilon s^T[t]	\big( \lambda^\star  - \lambda[t] \big)$ to expression \eqref{eq:LemmaDualGapExpansion}, leading to the following
\begin{align}
\big\|  & \lambda^\star -\lambda[t+1]\big\|^2  \leq
\big\|\lambda^\star-\lambda[t]\big\|^2
+\epsilon^2 \big\|\tilde{s}[t]\big\|^2 \nonumber\\
&+ 2\epsilon \big( s[t] - \tilde{s}[t]	\big)^T \big( \lambda^\star  - \lambda[t] \big) 
- 2\epsilon s^T[t]	\big( \lambda^\star  - \lambda[t] \big).
\end{align}
By applying the Cauchy-Schwarz inequality to the third term on the right hand side, we further rewrite the expression as
\begin{align}
\big\|  & \lambda^\star -\lambda[t+1]\big\|^2  \leq
\big\|\lambda^\star-\lambda[t]\big\|^2 
+\epsilon^2 \big\|\tilde{s}[t]\big\|^2 \nonumber\\
&+ 2\epsilon \big\| s[t] - \tilde{s}[t]	\big\| \big\| \lambda^\star  - \lambda[t] \big\|
- 2\epsilon s^T[t]	\big( \lambda^\star  - \lambda[t] \big).
\end{align}
Then, we can further bound the expression by relying on the $L$-Lipschitz continuity of the subgradients of the dual function
\begin{align}
\big\|  & \lambda^\star -\lambda[t+1]\big\|^2  \leq
\big\|\lambda^\star-\lambda[t]\big\|^2
+\epsilon^2 \big\|\tilde{s}[t]\big\|^2 \nonumber\\
&+ 2\epsilon L\big\| \lambda[t] - \tilde{\lambda}[t]	\big\| \big\| \lambda^\star  - \lambda[t] \big\|
- 2\epsilon s^T[t]	\big( \lambda^\star  - \lambda[t] \big) 
\end{align}
Now, we bound the difference $\big\| \lambda[t] - \tilde{\lambda}[t]	\big\| $ between the dual variables $\lambda[t]$ and their asynchronous counterpart $\tilde{\lambda}[t]$. Recall that Assumption \ref{assump:asynchronicity} states that there exists an asynchronicity limit of $B$ time slots between the global and asynchronous variables. This means that the asynchronous dual variables $\tilde{\lambda}[t]$ are, at most, $B$ subgradients steps out of synchronysm. We can then bound the difference by $\big\| \lambda[t] - \tilde{\lambda}[t]	\big\|  \leq \epsilon \big\| \sum_{l=t-B-1}^{t-1}\tilde{s}[l]\big\| \leq \epsilon \sum_{l=t-B-1}^{t-1}\big\| \tilde{s}[l]\big\|$, where we have also applied the triangle inequality. Then, we write
\begin{align}
\big\|  \lambda^\star -\lambda[t+1]\big\|^2  &\leq
\big\|\lambda^\star-\lambda[t]\big\|^2
+\epsilon^2 \big\|\tilde{s}[t]\big\|^2  \nonumber\\
&+ 2\epsilon^2 L\left( \sum\limits_{l=t-B-1}^{t-1}\big\| \tilde{s}[l]\big\|  \big\|  \lambda^\star - \lambda[t] \big\| \right)  \nonumber\\
&- 2\epsilon s^T[t]	\big( \lambda^\star  - \lambda[t] \big).
\end{align}
The third term on the right hand side can be further expanded by making use of the inequality $\|u\| \|v\| \leq \|u\|^2 + \|v\|^2$, leading to the following expression
\begin{align}
\big\|  \lambda^\star -\lambda[t+1]\big\|^2  &\leq
\big\|\lambda^\star-\lambda[t]\big\|^2
+\epsilon^2 \big\|\tilde{s}[t]\big\|^2 \nonumber\\
&+ 2\epsilon^2 L \sum\limits_{l=t-B-1}^{t-1}\biggl( \big\| \tilde{s}[l]\big\|^2 +  \big\|  \lambda^\star - \lambda[t] \big\|^2 \biggr) \nonumber\\
&- 2\epsilon s^T[t]	\big( \lambda^\star  - \lambda[t] \big).
\end{align}
Rearranging the terms 
\begin{align}
\big\|  \lambda^\star -\lambda[t+1]\big\|^2  &\leq
\big\|\lambda^\star-\lambda[t]\big\|^2
+  2\epsilon^2 L B \big\|  \lambda^\star - \lambda[t] \big\|^2 \nonumber\\
&+\epsilon^2 \big\|\tilde{s}[t]\big\|^2 
+ 2\epsilon^2 L \sum\limits_{l=t-B-1}^{t-1}\big\| \tilde{s}[l]\big\|^2 \nonumber\\
&- 2\epsilon s^T[t]	\big( \lambda^\star  - \lambda[t] \big) 
\end{align}
Then, separate the last term on the right hand side, and take $-\epsilon s^T[t]	\big( \lambda^\star  - \lambda[t] \big)$ and note that we can rewrite it as $-\epsilon \big(s^{\star} -s[t] \big)	^T\big( \lambda^\star  - \lambda[t] \big)$. Then, by strong concavity we can bound this term by $-\epsilon m \big\|  \lambda^\star - \lambda[t] \big\|^2$. Now, we take the expectation conditioned on $\lambda[t]$ on both sides of the previous inequality
\begin{align}
\E \big[ \|  \lambda^\star & -\lambda[t+1]\big\|^2 | \lambda[t] \big]   \leq
\bigl( 1- \epsilon m +  2\epsilon^2 L B \bigr)  \big\|  \lambda^\star - \lambda[t] \big\|^2  \nonumber\\
&+\epsilon^2 \E\bigl[\|\tilde{s}[t] \|^2| \lambda[t]\bigr]
+ 2\epsilon^2 L \sum\limits_{l=t-B-1}^{t-1} \E\bigl[\|\tilde{s}[l] \|^2| \lambda[t]\bigr] \nonumber\\
&-\epsilon \E \big[ s^T[t]	 | \lambda[t] \big]\big( \lambda^\star  - \lambda[t] \big) 
\end{align}
And then, by applying the subgradient bound given by $\E\bigl[\|s[t] \|^2| \lambda[t]\bigr] \leq S^2$, and  particularizing Proposition \ref{prop:avgSubgradient} with $\lambda=\lambda^\star$, we have
\begin{align}
\E \big[ \|  \lambda^\star  - & \lambda[t+1]\big\|^2 | \lambda[t] \big]   \leq
\bigl( 1- \epsilon m +  2\epsilon^2 L B \bigr)  \big\|  \lambda^\star - \lambda[t] \big\|^2 \nonumber\\
&+\epsilon^2 S^2 + 2\epsilon^2 L B S^2 
- \epsilon \big( g(\lambda^\star) - g(\lambda[t]) \big),
\end{align}
which gives us the desired inequality.
\end{proof}
The previous lemma holds on average, while we are interested in establishing convergence almost surely. We leverage on this lemma and resort to a supermartingale argument to show that Algorithm 1 converges to a neighborhood of the optimal solution of the dual function.
\begin{lemma}
\label{lmm:DualGap}
Let $\E\bigl[\|s[t] \|^2| \lambda[t]\bigr] \leq S^2$  be a bound on the second moment of the norm of the stochastic subgradients $s[t]$. Further, consider the dual updates of Algorithm \ref{alg:Algorithm}, with step size $\epsilon \leq m/(2LB)$. Then, assume that the dual variable $\lambda[T]$ is given for an arbitrary time $T$ and define as $\lambda_{\mathrm{best}}[t] := \argmin_{\lambda[l]} g(\lambda[l])$ the dual variable leading to the best value of the of the dual function for the interval $l \in [T,t]$. Then, we have
\begin{align}
\label{eq:LmmDualGapEQ}
\lim\limits_{t \to \infty} g(\lambda_{\mathrm{best}}[t]|\lambda[T]) \geq g(\lambda^{\star}) - \epsilon^2 S^2 - 2\epsilon^2 L B S^2 ~ \text{a.s.}
\end{align}
\end{lemma}
\begin{proof}
Let $T=0$ for simplicity of exposition. Then, define the sequence $\alpha[t]$ corresponding to a stopped process tracking the dual distance $\big\|\lambda^\star-\lambda[t] \big\|^2$ until the optimality gap $g(\lambda^\star) - g(\lambda[t])$ falls bellow $\epsilon S^2 + 2\epsilon L B S^2$. Namely, 
\begin{align}
\alpha [t] :&=
\big\|\lambda^\star-\lambda[t] \big\|^2 \nonumber\\
 &\mathbb{I}\left\{  g(\lambda^{\star}) - g(\lambda_{\mathrm{best}}[t])> \epsilon S^2 + 2\epsilon L B S^2\right\},
 \label{eq:LmmDualGapDefAlpha}
\end{align}
where $\mathbb{I} \{ \cdot \}$ is the indicator function. In a similar manner define the sequence $\beta [t]$ as follows
\begin{align}
\beta [t] :&=
\biggl(
\epsilon \big( g(\lambda^\star) - g(\lambda[t]) \big)
-\epsilon^2 S^2 - 2\epsilon^2 L B S^2 
\biggr) \nonumber\\
 &\mathbb{I}\left\{  g(\lambda^{\star}) - g(\lambda_{\mathrm{best}}[t])> \epsilon S^2 + 2\epsilon L B S^2\right\}.
 \label{eq:LmmDualGapDefBeta}
\end{align}
Now, let $\mathcal{F}[t]$ be the filtration measuring $\alpha[t]$, $\beta[t]$ and  $\lambda[t]$. Since $\alpha[t]$ and $\beta[t]$ are completely determined by $\lambda[t]$, and $\lambda[t]$ is a Markov process, conditioning on $\mathcal{F}[t]$ is equivalent to conditioning on $\lambda[t]$. Hence, we can write the expectation $\E\left[ \alpha[t] | \mathcal{F}[t] \right] = \E\left[ \alpha[t] | \lambda[t] \right]$. Now, consider this expectation at time $t+1$, given by
\begin{align}
\E & \bigl[  \alpha [t+1]| \lambda[t] \bigr] = \E\bigl[
\big\|\lambda^\star-\lambda[t+1] \big\|^2 \nonumber\\
 &\mathbb{I}\left\{  g(\lambda^{\star}) - g(\lambda_{\mathrm{best}}[t+1])> \epsilon S^2 + 2\epsilon L B S^2\right\}
| \lambda[t] \bigr].
\end{align}
By noting that the indicator term is lower or equal than 1, we can upper bound this expression as
\begin{align}
\E\bigl[ & \alpha [t+1]| \lambda[t] \bigr] \leq \E\bigl[
\big\|\lambda^\star-\lambda[t+1] \big\|^2 | \lambda[t] \bigr].
\end{align}
Then, by application of Lemma \ref{lmm:AverageDualDist} we have
\begin{align}
\E\bigl[ & \alpha [t+1]| \lambda[t] \bigr] \leq
\bigl( 1- \epsilon m +  2\epsilon^2 L B \bigr)  \big\|  \lambda^\star - \lambda[t] \big\|^2 \nonumber\\
&+\epsilon^2 S^2 + 2\epsilon^2 L B S^2 
- \epsilon \big( g(\lambda^\star) - g(\lambda[t]) \big).
\end{align}
By making use of the definitions of $\alpha[t]$ and $\beta[t]$, given by equations \eqref{eq:LmmDualGapDefAlpha} and \eqref{eq:LmmDualGapDefBeta}, we can rewrite the previous expression as
\begin{align}
\E\bigl[ \alpha [t+1]| \lambda[t] \bigr] \leq \bigl( 1- \epsilon m +  2\epsilon^2 L B \bigr) \alpha[t] - \beta[t].
\end{align}
Now, since the step size chosen is $\epsilon \leq m/(2LB)$, this means that the factor multiplying the process $\alpha[t]$ is upper bounded by $ \bigl( 1- \epsilon m +  2\epsilon^2 L B \bigr) \leq 1$. Therefore, we can further write the expectation as 
\begin{align}
\E\bigl[ \alpha [t+1]| \lambda[t] \bigr] \leq \alpha[t] - \beta[t].
\end{align}
Since by definition the processes $\alpha [t]$ and $\beta [t]$ are nonnegative, the supermartingale convergence theorem \cite[Theorem 5.2.9]{durrett2010probability} states that the sequence $\alpha [t]$ converges almost surely, and the sum $\sum_{t=1}^{\infty}\beta[t] < \infty$ is almost surely finite. This carries the implication that $\lim\inf_{t \to \infty} \beta[t] = 0$. Given the definition of the sequence $\beta[t]$, this implies that $\lim_{t \to \infty} g(\lambda_{\mathrm{best}}[t]) \geq g(\lambda^{\star}) - \epsilon^2 S^2 - 2\epsilon^2 L B S^2$ almost surely.
\end{proof}

Now, it suffices to show that the iterates generated by the algorithm are almost surely feasible to the original problem \eqref{eq:optProbNoAux} if the constants $\bar{\nu}_{ij}$ are chosen to be an upper bound of the optimal $\nu^\star_{ij}$ multipliers.

\begin{proposition}[Auxiliary Feasibility]
\label{prop:AuxFeasibility}
Assume there exist feasible primal variables $z_i, s_{ij}$ and $y_{ij}$, such that for some $\xi > 0$, we have $\log(p_i) - \log(s_{ii}) - \sum_{j\neq i}\log\left(1-s_{ij}\right) < -\xi$, $s_{ii} - \E q_i z_i - y_{ii} <  -\xi$, $\E q_c z_j - y_{ij} - s_{ij} < -\xi$ and $\E z_i - \E e_i < -\xi$. Then, the sequences generated by Algorithm \ref{alg:Algorithm}, satisfy the constraints \eqref{eq:optProbAuxC0}$-$\eqref{eq:optProbAuxC3} almost surely.
\end{proposition}
\begin{proof}
First, we start by upper bounding the value of the dual variables. We collect the feasible primal variables in a vector $\hat{z}=\{z_i, s_{ij},y_{ij}\}$. Then, given feasible primal variables $\hat{z}$ we bound the value of the dual function $g(\lambda)$. Recall that the dual function is defined as $g(\lambda)=\min_{z \in \mathcal{X}} \mathcal{L}(z,\lambda)$, then for any feasible  $\hat{z}$, we necessarily have $g(\lambda) \leq\mathcal{L}(\hat{z},\lambda)$. Hence, we can write
\begin{align}
g &(\lambda)\leq \sum\displaystyle_{i=1}^{M} \E z_i^2 + \sum\displaystyle_{i=1}^{M} \sum\displaystyle_{j=1}^{M} \E \bar{\nu}_{ij} y_{ij} \nonumber\\
&+ \sum\displaystyle_{i=1}^{M} \phi_i\bigl(\log\left(p_i\right) - \log\left(s_{ii}\right) - \sum\displaystyle_{j\neq i}\log\left(1-s_{ij}\right) \bigr) \nonumber\\
&+ \sum\displaystyle_{i=1}^{M} \nu_{ii}\bigl(s_{ii} - \E q_i z_i - y_{ii}\bigr)  \nonumber\\
&+ \sum\displaystyle_{i=1}^{M} \sum\displaystyle_{j \neq i} \nu_{ij}\bigl(\E q_c z_j - y_{ij} - s_{ij}\bigr) \nonumber\\
&+ \sum\displaystyle_{i=1}^{M}\beta_i\bigl(\E z_i - \E e_i\bigr).
\end{align}
Since we have a constant $\xi > 0$ such that $\log(p_i) - \log(s_{ii}) - \sum_{j\neq i}\log\left(1-s_{ij}\right) < -\xi$, $s_{ii} - \E q_i z_i - y_{ii} <  -\xi$, $\E q_c z_j - y_{ij} - s_{ij} < -\xi$ and $\E z_i - \E e_i < -\xi$. We can simplify the bound on the dual function to the following inequality
\begin{align}
g(\lambda) \leq  
\sum_{i=1}^{M} \E z_i^2 + \sum_{i=1}^{M} \sum_{j=1}^{M} \E \bar{\nu}_{ij} y_{ij} - \xi \lambda^T1.
\end{align}
Then, we simply reorder the previous expression to establish an upper bound on the dual variables,
\begin{align}
 \lambda \leq  
\frac{1}{\xi}
\biggl(\sum_{i=1}^{M} \E z_i^2 + \sum_{i=1}^{M} \sum_{j=1}^{M} \E \bar{\nu}_{ij} y_{ij} - g(\lambda) \biggr),
\end{align}
where the inequality is taken component-wise for all elements of the vector $\lambda$ with respect to the scalar on the right hand side of the inequality. Then, by Lemma \ref{lmm:DualGap} we can certify the existence of a time $t \geq T_1$ such that $g(\lambda[t]) \geq g(\lambda^{\star}) - \epsilon^2 S^2 - 2\epsilon^2 L B S^2$. Hence, we can write
\begin{align}
 \lambda[t] \leq  
\frac{1}{\xi}
\biggl(\sum_{i=1}^{M} \E z_i^2 +& \sum_{i=1}^{M} \sum_{j=1}^{M} \E \bar{\nu}_{ij} y_{ij} \nonumber\\
& -g(\lambda^{\star}) + \epsilon^2 S^2 + 2\epsilon^2 L B S^2 
\biggr)
\label{eq:TheoFeasibilityDualVarBound}
\end{align}
for some $t\geq T_1$. Now, consider the feasibility conditions of the optimization problem with the auxiliary variables \eqref{eq:optProbAuxC0}$-$\eqref{eq:optProbAuxC3}, which are given by the long term behavior of the following inequalities
\begin{align}
&\lim_{t \to \infty}\frac{1}{t} \sum\limits_{l=1}^{t}
\bigl( \log(p_i) - \log(s_{ii}[l]) - \sum\limits_{j\neq i}\log\left(1-s_{ij}[l]\right) \bigr) \leq 0 \label{eq:TheoFeasibilityC1}\\
&\lim_{t \to \infty}\frac{1}{t} \sum\limits_{l=1}^{t}
\bigl( s_{ii}[l] - q_i[l] z_i[l] - y_{ii}[l] \bigr) \leq 0  \label{eq:TheoFeasibilityC2}\\
&\lim_{t \to \infty}\frac{1}{t} \sum\limits_{l=1}^{t}
\bigl(  q_c z_j[l] - y_{ij}[l] - s_{ij}[l] \bigr) \leq 0  \label{eq:TheoFeasibilityC3}\\
&\lim_{t \to \infty}\frac{1}{t} \sum\limits_{l=1}^{t} 
\bigl( z_i[l] - e_i[l] \bigr) \leq 0.  \label{eq:TheoFeasibilityC4}
\end{align}
These inequalities simply correspond to the stochastic subgradients of the optimization problem \eqref{eq:optProbAux}. Therefore, we can write the feasibility conditions in a condensed form as $\lim_{t \to \infty}\frac{1}{t} \sum_{l=1}^{t} s[l] \leq 0$. Now, we consider the dual updates of the problem, given by $\lambda[t+1]=\bigl[\lambda[t]+\epsilon \tilde{s}[t] \bigr]^+$. Since the projection is nonnegative, we can upper bound $\lambda[t+1]$ by 
\begin{align} 
\lambda[t+1]\geq \lambda[t] + \epsilon \tilde{s}[t] \geq \lambda[1] +\epsilon \sum_{l=1}^{t} \tilde{s}[l] \geq \epsilon\sum_{l=1}^{t} \tilde{s}[l],
\label{eq:TheoFeasibilityDualUpdateBound}
\end{align}
where we have removed the projection and further upper bounded the expression by recursively substituting the dual updates. Now, we proceed to prove feasibility of the constraints of the auxiliary problem \eqref{eq:optProbAux}. In order to to this, we follow by contradiction. Assume that the conditions \eqref{eq:TheoFeasibilityC1}$-$\eqref{eq:TheoFeasibilityC4} are infeasible. This means that there exists some time $T_2$ where there is a constant $\delta > 0$ such that for $t\geq T_2$ we have $\lim_{t \to \infty}\frac{1}{t} \sum_{l=1}^{t} s[l] \geq \delta$. Substituting this expressions in the dual update bound \eqref{eq:TheoFeasibilityDualUpdateBound} we have that $\lambda[t+1]\geq \epsilon \delta t$. Then, we can freely choose a time $t \geq T_2$ such that
\begin{align}
 \lambda[t] >  
\frac{1}{\xi}
\biggl(\sum_{i=1}^{M} \E z_i^2 +& \sum_{i=1}^{M} \sum_{j=1}^{M} \E \bar{\nu}_{ij} y_{ij} \nonumber\\
& -g(\lambda^{\star}) + \epsilon^2 S^2 + 2\epsilon^2 L B S^2 
\biggr)  
\end{align}
However, the upper bound established in \eqref{eq:TheoFeasibilityDualVarBound} contradicts this expression. Therefore, the inequalities \eqref{eq:TheoFeasibilityC1}$-$\eqref{eq:TheoFeasibilityC4} are satisfied almost surely, which implies that the constraints \eqref{eq:optProbAuxC0}$-$\eqref{eq:optProbAuxC3} of the auxiliary optimization problem \eqref{eq:optProbAux} are almost surely satisfied. 
\end{proof}
This proposition allows us to certify that the constraints of the problem with the auxiliary variables are satisfied. However, this does not ensure stability. Nonetheless, we show that if the parameters $\bar{\nu}_{ij}$ are chosen as to upper bound the optimal dual variables $\nu_{ij}$, then the optimal auxiliary variables are zero, and the proposed algorithm also satisfies the constraints \eqref{eq:optProbNoAuxC0}$-$\eqref{eq:optProbNoAuxC3} of the original problem without the auxiliary variables.
\begin{theorem}[Stability]
\label{theo:stability}
Assume there exist feasible primal variables $z_i$ and $s_{ij}$, such that for some $\xi > 0$, we have $\log(p_i) - \log(s_{ii}) - \sum_{j\neq i}\log\left(1-s_{ij}\right) < -\xi$, $s_{ii} - \E q_i z_i <  -\xi$, $\E q_c z_j - s_{ij} < -\xi$ and $\E z_i - \E e_i < -\xi$. Further, let $\bar{\nu}_{ij}$ be an upper bound to the optimal $\nu_{ij}$ multipliers. Then, the scheduling decisions $z_i[t]$ generated by Algorithm \ref{alg:Algorithm} satisfy the successful packet transmission requirement 
\begin{align}
\lim\limits_{t \to \infty} \frac{1}{t}\sum\limits_{l=1}^{t} q_i[l] z_i[l]\prod\limits_{j\neq i}\left(1-q_c z_j[l]\right) > p_i,
\end{align}
which ensures the stability of all control loops.
\end{theorem}
\begin{proof}
To verify this, subtract the Lagrangian of the optimization problem with the auxiliary variables \eqref{eq:optProbAux} and the original problem \eqref{eq:optProbNoAux}. This difference is given by,
\begin{align}
\mathcal{L}(z,\lambda) - \hat{\mathcal{L}}(z,\lambda) &=
\sum_{i=1}^{M} \sum_{j=1}^{M} \bigl( \bar{\nu}_{ij} - \nu_{ij} - \theta_{ij} +\mu_{ij}  \bigr) y_{ij} \nonumber\\
&- \sum_{i=1}^{M} \sum_{j=1}^{M} \mu_{ij} \bar{y}_{ij}
\label{eq:LagrangianDiff}
\end{align}
where $\theta_{ij} \geq 0$ and $\mu_{ij} \geq 0$ are the Lagrange multipliers of the implicit constraints $y_{ij} \geq 0$ and $y_{ij} \leq \bar{y}_{ij}$, respectively. To certify the equivalence between the optimization problems \eqref{eq:optProbNoAux} and \eqref{eq:optProbAux} we need to certify that \eqref{eq:LagrangianDiff} is zero for the optimal values. This implies that there must exist Lagrange multipliers such that $\bar{\nu}_{ij} - \nu_{ij} - \theta_{ij} +\mu_{ij}=0$ and $\mu_{ij}=0$ for all $i,j$. Since $\nu^\star_{ij} \leq \bar{\nu}_{ij}$, we can find multipliers satisfying the constraints by letting $\mu^\star_{ij}=0$ and $\theta^\star_{ij} = \bar{\nu}_{ij} - \nu^\star_{ij} $. Then, $\mathcal{L}(z,\lambda) - \hat{\mathcal{L}}(z,\lambda) = 0$, which implies the solution of both problems is equal. Since $\lim_{t \to \infty}\frac{1}{t} \sum_{l=1}^{t}y_{ij}[l]=y^\star_{ij}$ and $y^\star_{ij}=0$, the primal variables $z_i$ and $s_{ij}$ almost surely satisfy the original constraints of the optimization problem without the auxiliary variable, given by \eqref{eq:optProbNoAuxC0}$-$\eqref{eq:optProbNoAuxC3}. Since constraints \eqref{eq:optProbNoAuxC0}$-$\eqref{eq:optProbNoAuxC2} are equivalent to the constraint $p_i < \E \bigl[ q_i z_i \prod_{j\neq i}\left(1-q_c z_j\right) \bigr]$, by Proposition \ref{prop:ControlAbstraction} we have that Algorithm \ref{alg:Algorithm} generates schedules $z_i[t]$ that ensure the stability of all control loops.
\end{proof}
Theorem \ref{theo:stability} states that if there exist schedules $z_i[t]$ capable of stabilizing the plants, then Algorithm \ref{alg:Algorithm} generates them. Note that the optimal auxiliary variables $y_{ij}$ take the zero value and are therefore not necessary in the long run to satisfy the constraints of the original optimization problem. As previously discussed, they have been introduced in order to provide a way to bound the dual variables and allow us to enforce causality in the energy consumption.

\section{Numerical Results}
\label{sec:NumRes}

In this section, we study the performance of the proposed random access scheme with energy harvesting sensors. We consider a scalar control system, with $M=2$ plants over $T=10{,}000$ time slots. The plant dynamics are given by $A_{o,1}=1.1$ and $A_{c,1}=0.15$ for the first plant, and $A_{o,2}=1.05$ and $A_{c,2}=0.1$ for the second one. Hence, the first system is slightly more unstable than the second one. Further, we consider both systems to be perturbed by i.i.d. zero-mean Gaussian noise. Also, we assume the same performance requirement for both plants, given by the Lyapunov function $V_i(x_i[t])=x_i^2[t]$ and an expected decrease rate of $\rho_i=0.8$. With regards to the communication aspects, we consider a system where the channel states $h_i[t]$ are i.i.d. exponential variables with mean $\E \bigl[ h_i[t] \bigr]=2$, and the decoding probability $q(h_i[t])$ is given by the function shown in Figure \ref{fig:decodingFun}. Since the communication medium is shared, we consider that packet collisions occur with probability $q_c=0.25$. Moreover, we consider the sensing devices to be powered by an energy harvesting process of rate $\E \bigl[ e_i \bigr]=0.5$, and that they store the collected energy in batteries of size $b_i^{\max}=20$. Finally, the parameters of the algorithm are chosen as $\bar{y}_{ij}=25$, $\bar{\nu}_{ij}=19$, and step size $\epsilon=1$. 

\subsection{System Dynamics}

\begin{figure}[t]
	\centering
    \includegraphics[width=\columnwidth]{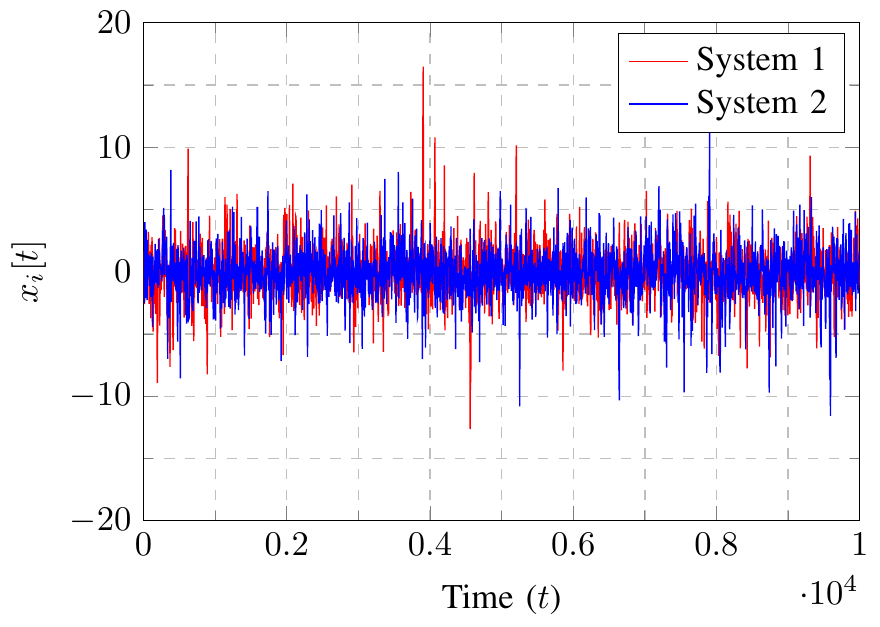} 
	\caption{Evolution of the plant state at each time slot.}
	\label{fig:SystemState}
\end{figure}

We start by studying the evolution of the system dynamics. In Figure \ref{fig:SystemState}, we plot the evolution of the plant state at each time slot. As expected, since our proposed policy stabilizes both plants, the system state oscillates around the zero value. Furthermore, this figure illustrates that System 1 is slightly more unstable than System 2. This is evidenced by the somewhat more pronounced peaks of instability, and higher variance of the plant state $x_1[t]$. In a similar manner, this behavior is also shown in Fig. \ref{fig:BatteryState}. In this figure, we have plotted the evolution of the battery state of the sensing devices of both systems. Since the first plant is slightly more unstable than the second plant, the energy consumption of the sensor in the first system is slightly more pronounced. Also, by taking a look at this figure, one can expect the first system to have larger battery requirements than the second system. Intuitively, since System 1 is more unstable, its sensor node requires a larger battery to counteract its instability. The extent of this requirement will become more apparent once we take a look at the values of the dual variables.

\begin{figure}[t]
	\centering
    \includegraphics[width=\columnwidth]{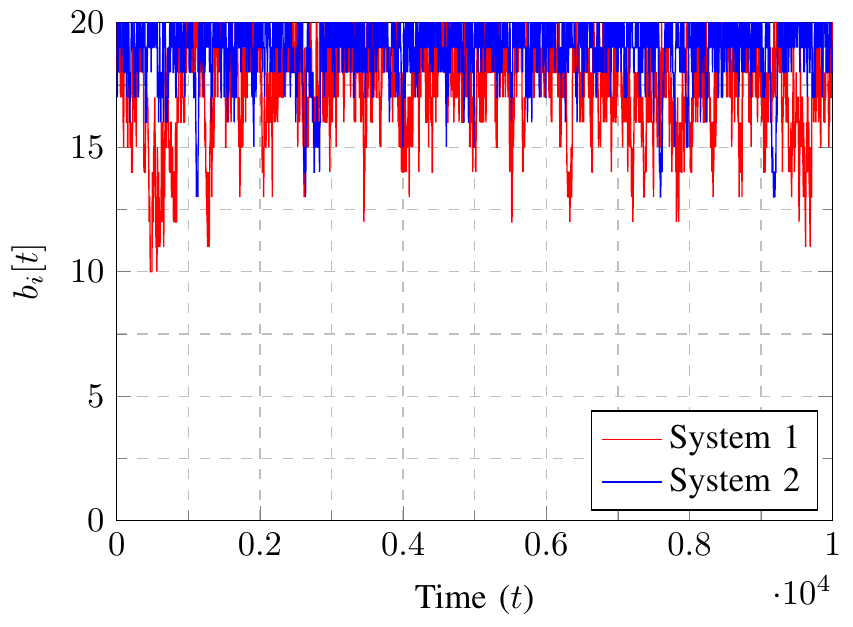} 
	\caption{Energy stored in the batteries at each time slot.}
	\label{fig:BatteryState}
\end{figure}

\subsection{Stability}

\begin{figure}[t]
	\centering
    \includegraphics[width=\columnwidth]{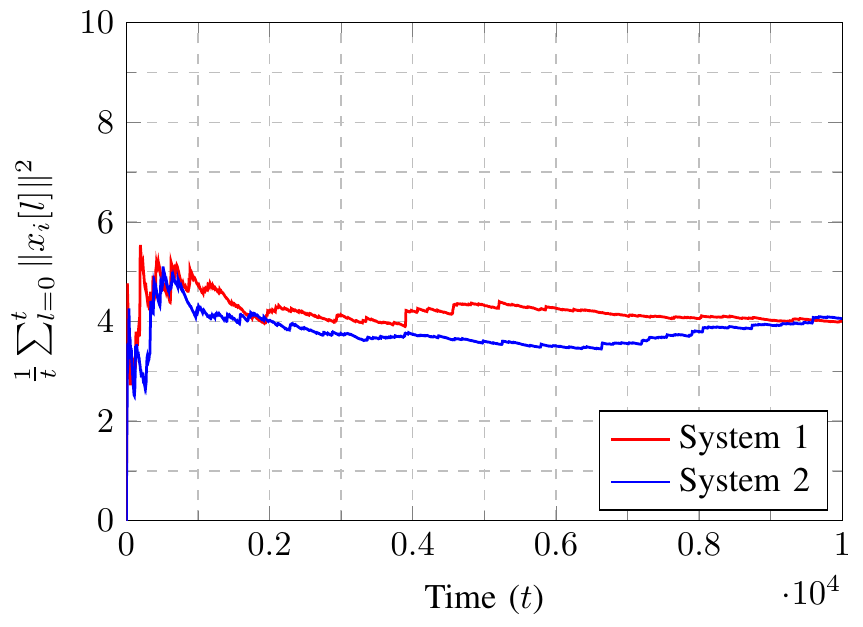} 
	\caption{Evolution of the system control performance.}
	\label{fig:SystemControl}
\end{figure}

In order to study the stability of the plants under our proposed policy, we look at the long term evolution of the system states. First, we look at the evolution of the system control performance. By our design, we require the Lyapunov function $V_i(x_i[t])=x_i^2[t]$ to decrease at an average rate of $\rho_i=0.8$. Iterating expression \eqref{eq:ControlPerfDecrease} in Proposition \ref{prop:ControlAbstraction}, we have that the limit of the control performance is upper bounded in the long run by the term $\tr(P_i W_i) / (1-\rho_i)$. By particularizing this expression to our parameters, we expect the control performance in the limit to be below $\tr(P_i W_i) / (1-\rho_i)=1/(1-0.8)=5$. We plot in Figure \ref{fig:SystemControl} the system control evolution. As expected, both systems are asymptotically stable and the control performance converges to approximately $(1/T)\sum_{t=0}^{T}\|x_i[t]\|^2 \approx 4$ for both plants.

\begin{figure}[t]
	\centering
    \includegraphics[width=\columnwidth]{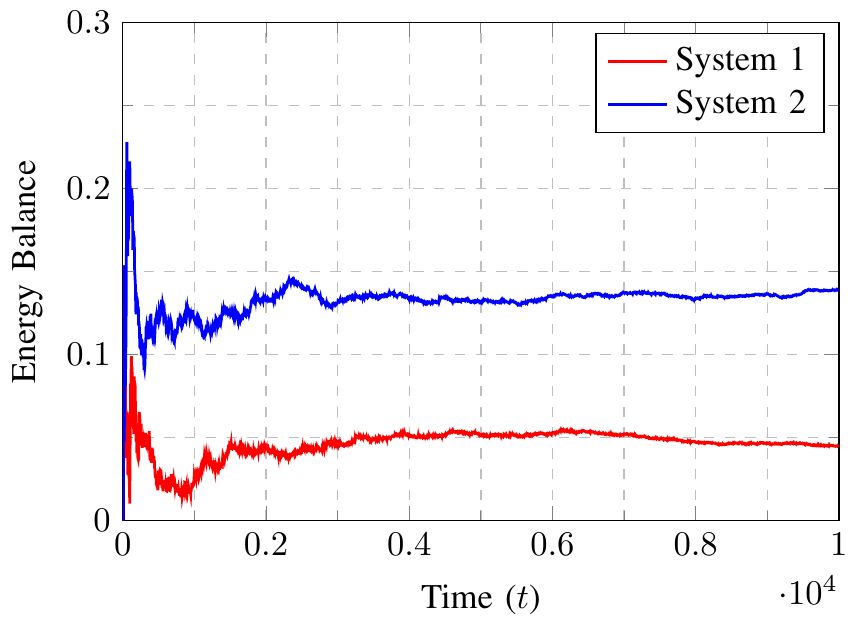} 
	\caption{Average energy balance in the system over time. Given by the expression $(1/t)\sum_{l=0}^{t}\left( e_i[l] -  z_i[l] \right)$.}
	\label{fig:EnergyBalance}
\end{figure}

\begin{figure}[t]
	\centering
    \includegraphics[width=\columnwidth]{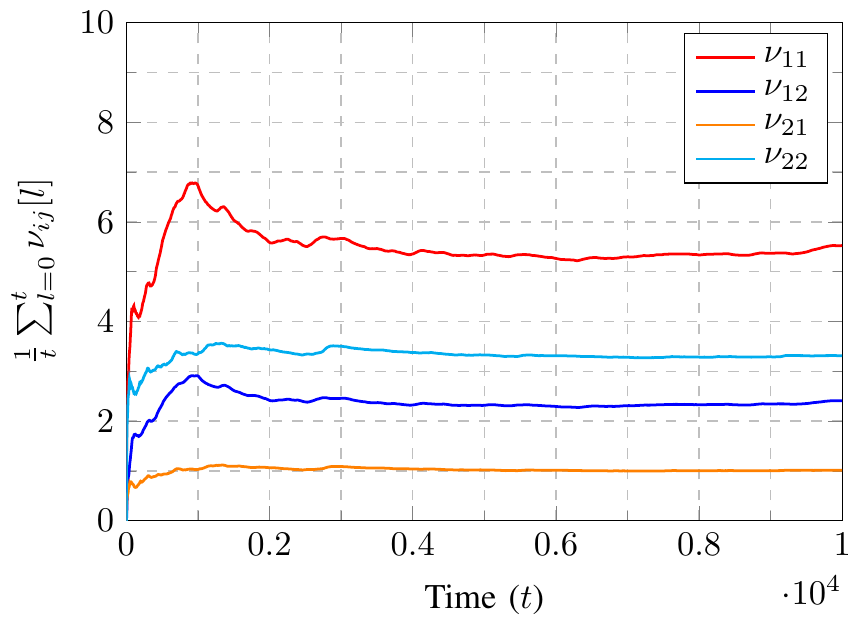} 
	\caption{Average evolution of the dual variables $\nu_{ij}$ over time.}
	\label{fig:DualVariables}
\end{figure}

Another interesting measure to consider is the energy balance of the systems. We denote by energy balance the difference between the harvested energy and the consumed energy. Thus, the average energy balance of the $i$-th system at time $t$ is given by $(1/t)\sum_{l=0}^{t}\left( e_i[l] -  z_i[l] \right)$. We plot this measure in Figure \ref{fig:EnergyBalance}. As expected, since both sensors are powered by energy harvesting processes of the same mean $\E \bigl[ e_i \bigr]=0.5$ and System 1 is more unstable than System 2, the energy balance of the first system is lower. Also, note that the lower bound on the energy balance is zero, since the total energy spent has to necessarily be lower or equal to the total energy harvested. This allows us to interpret the energy balance as a measure of how much more control performance can be obtained with the same energy harvesting process. For example, System 1 has an energy balance of approximately $0.05$ units. Since we have assumed an unitary power consumption, this means that System 1 has energy to support an increase by $0.05$ of its transmit probability. In the same manner, System 2 can support an increase of around $0.14$ of its transmit probability.

Now, we take a look at the dual variables. As we have thoroughly discussed in previous sections, the selection of the $\bar{\nu}_{ij}$ parameters is crucial to the proper operation of the algorithm. Specifically, these parameters should be chosen such that their corresponding optimal dual variables are upper bounded by them. In Figure \ref{fig:DualVariables} we plot the time averaged dual variables $\nu_{ij}$, where the average over time of these variables converges to their optimal value. First, the choice of $\bar{\nu}_{ij}=19$ for all $i,j$ satisfies the required condition, i.e., being an upper bound of the optimal dual variables. Further, by taking a closer look at these dual variables, we can gain some insight into the behavior of the system. First, note that the variables $\nu_{ii}$ are associated to the constraint $s_{ii} \leq \E q_i z_i + y_{ii}$, and hence, represent the requirement of plant $i$ to transmit its state. By looking at expression \eqref{eq:primalZClosed}, corresponding to the closed-form solution of the primal $z_i[t]$, a larger value of the dual variable $\nu_{ii}[t]$ leads to a larger value of the scheduling variable $z_i[t]$. Thus, since System 1 is more unstable than System 2, we have that $\nu_{11} > \nu_{22}$. In a similar way, the variables $\nu_{ij}$ for $j \neq i$ are associated to the constraint $s_{ij} \geq \E q_c z_j - y_{ij}$ and represent, at node $i$, the interference-adjusted need to transmit of the other plants $j \neq i$. Therefore, in our two-plant scenario, a large value of $\nu_{21}[t]$ leads to a lower $z_1[t]$. And in a similar manner as previously, since System 1 is more unstable than System 2, and the systems interfere symmetrically, we have that $\nu_{12} > \nu_{21}$. Also, when previously evaluating Fig. \ref{fig:BatteryState}, we expected System 1 to have higher battery requirements due to its higher instability. Now, as per Proposition \ref{prop:EnergyCausality}, which states that the required battery size $b_i^{\max}$ is proportional to the dual variables $\bar{\nu}_{ii}$ by the inequality $b_i^{\max} \geq \frac{1}{\epsilon} \bar{\nu}_{ii} + 1$, we see that System 1 requires a larger $\bar{\nu}_{ii}$ value, and hence, a larger battery.

\subsection{Communication}

\begin{figure}[t]
	\centering
    \includegraphics[width=\columnwidth]{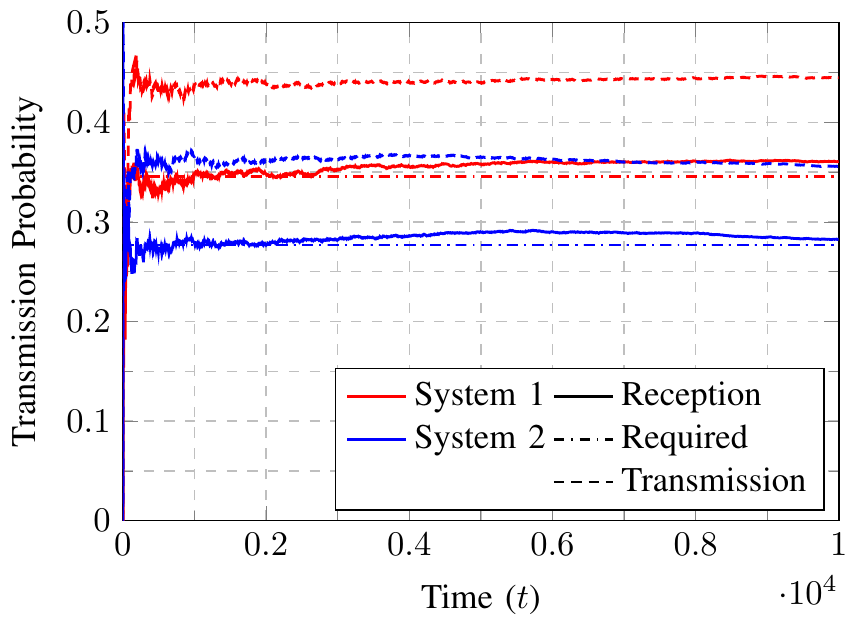} 
	\caption{Average transmission probabilities.}
	\label{fig:ProbTX}
\end{figure}

We turn our attention to the communication aspects of the proposed policy. As we have discussed previously, System 1 is slightly more unstable than System 2. Since we are requiring  for both plants an expected decrease rate of $\rho_i=0.8$ for a Lyapunov function $V_i(x_i[t])=x_i^2[t]$, by Proposition \ref{prop:ControlAbstraction} this translates to required successful transmission probabilities of $p_1 \approx 0.3453$ and $p_2 \approx 0.2769$, respectively. As expected, the less stable system requires a higher successful transmission probability. In Figure \ref{fig:ProbTX} we have plotted the resulting transmission probabilities during our simulation. We look at three different probabilities, (i) the required transmission probabilities, given by $p_i$; (ii) the actual transmission probabilities, $p_i^{TX} \triangleq (1/t)\sum_{l=0}^{t}z_i[l]$; and (iii) the successful reception probabilities, given by $p_i^{RX} \triangleq (1/t)\sum_{l=0}^{t}\bigl(q_i[l] z_i[l] \prod_{j\neq i}\left(1- q_c z_j[l] \right) \bigr)$.

While the required probabilities are $p_1 \approx 0.3453$ and $p_2 \approx 0.2769$, we have that the actual successful reception probabilities are $p_1^{RX} = 0.3607$ and $p_2^{RX} = 0.2827$. These probabilities are over the required ones and, hence, ensure the stability of the systems. However, the probability at which the sensors try to access the medium are higher, $p_1^{TX} = 0.4446$ and $p_2^{TX} = 0.3558$, respectively. This is due to the effects of the transmission medium. Packets can be lost if collisions occur, and they might not be decoded properly if the channel conditions are not sufficiently favorable (cf. Figure \ref{fig:decodingFun}).

The overall effect of the transmission medium is better displayed in Figure \ref{fig:Collision}, where we show the transmission schedules  from $t=1050$ to $t=1100$. In this figure, the bars represent the probability of successful decoding for a given time slot, and the stems represent an access to the medium. Further, collisions are represented by a red dot. From this plot, it is clear that the sensor node tends to access the medium when the channel conditions are favorable (i.e., the decoding function $q_i[t]$ takes values closer to $1$). Also, this figure evidences that collisions happen with a sufficiently low chance, since due to the independence assumption of the channel states between sensors, access does not happen at the same time very often.

\begin{figure}[t]
    \includegraphics[width=\columnwidth]{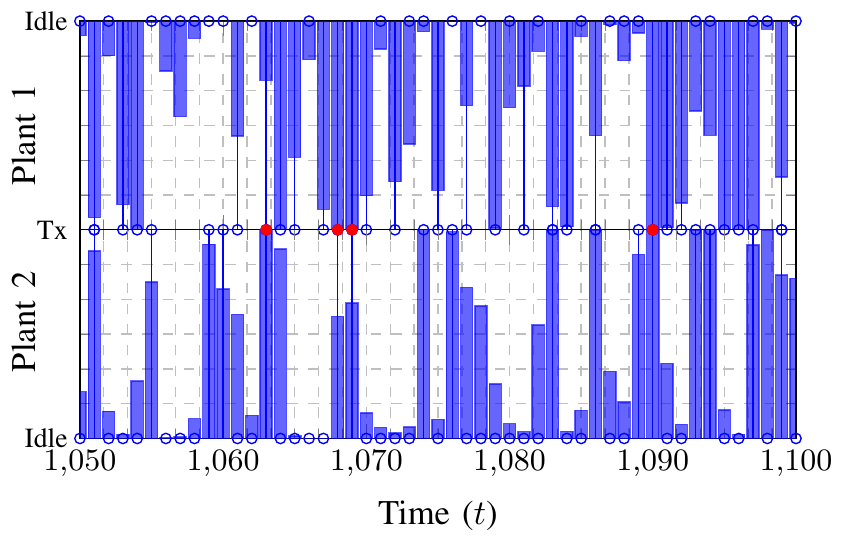} 

	\caption{Transmission schedules from $t=1050$ to $t=1100$. Bars represent the decoding probability $q_i[t]$, stems denote a transmission taking place, and a red dot denotes the occurrence of a collision.}
\label{fig:Collision}
\end{figure}

\section{Conclusions}
\label{sec:Conclusions}

In this work, we have designed a random access communication scheme for energy harvesting sensors in wireless control systems. We have considered a scenario with multiple plants sharing a wireless communication medium. Under these conditions, we have shown that the optimal scheduling decision is to transmit with a certain probability, which is adaptive to the time-varying channel, battery and plant conditions. In order to compute the optimal policy, we have provided an algorithm based on a stochastic dual method. The proposed algorithm is decentralized, where the sensors only need to share some of their dual variables. Furthermore, we have provided theoretical guarantees on the stability of all control loops under the proposed policy, including the consideration of asynchroniticty between the information shared between the nodes. Finally, we have validated by means of simulations the performance of the proposed scheme. The numerical results show that the random access policy is capable of stabilizing all control loops while also satisfying the energy constraints imposed by the energy harvesting process.

\bibliographystyle{IEEEtran}
\bibliography{bib}

\end{document}